\documentclass[11pt,twoside,reqno]{amsart}
\usepackage{geometry}
\usepackage[latin1]{inputenc}
\usepackage[OT1]{fontenc}
\usepackage{amsmath}
\usepackage{amsthm}
\usepackage{amssymb}
\usepackage[all]{xy}
\usepackage{bbm}
\usepackage{amscd}
\usepackage{a4wide}
\usepackage{enumitem}

\setitemize{itemsep=0.1cm}

\usepackage{stmaryrd}
\usepackage{mathtools}

\setenumerate[0]{label=(\roman*)}

\DeclareMathOperator{\Spec}{\mathsf{Spec}}

\DeclareMathOperator{\pr}{\mathsf{pr}}

\DeclareMathOperator{\id}{\mathsf{id}}

\DeclareMathOperator{\Hom}{\mathsf{Hom}}

\DeclareMathOperator{\sHom}{\mathcal{H}\textit{om}}
\DeclareMathOperator{\Ext}{\mathsf{Ext}}

\DeclareMathOperator{\Coh}{\mathsf{Coh}}

\DeclareMathOperator{\Ho}{\mathsf H}

\DeclareMathOperator{\Jac}{\mathsf{Jac}}

\DeclareMathOperator{\codim}{codim}

\DeclareMathOperator{\rank}{\mathsf{rank}}

\DeclareMathOperator{\triv}{\mathsf{triv}}

\DeclareMathOperator{\im}{\mathsf{im}}

\DeclareMathOperator{\cone}{\mathsf{cone}}

\DeclareMathOperator{\Sym}{\mathsf{Sym}}

\let\codim\relax
\DeclareMathOperator{\codim}{\mathsf{codim}}

\let\det\relax
\DeclareMathOperator{\det}{\mathsf{det}}

\newcommand{\D}{{\mathsf D}}

\DeclareMathOperator{\Aut}{\mathsf{Aut}}

\DeclareMathOperator{\Pic}{\mathsf{Pic}}

\newcommand{\cH}{{\mathcal H}}

\newcommand{\cJ}{{\mathcal J}}

\newcommand{\cU}{{\mathcal U}}
\newcommand{\cW}{{\mathcal W}}

\newcommand{\pt}{\mathsf{pt}}

\newcommand{\kk}{\mathsf{k}}

\newcommand{\IP}{\mathbb{P}}

\newcommand{\IZ}{\mathbb{Z}}

\DeclareMathOperator{\VB}{\mathsf{VB}}
\DeclareMathOperator{\FM}{\mathsf{FM}}

\DeclareMathOperator{\AJ}{\mathsf{AJ}}
\DeclareMathOperator{\Cent}{\mathsf{Cent}}

\makeatletter
\newcommand{\leqnomode}{\tagsleft@true}
\newcommand{\reqnomode}{\tagsleft@false}
\makeatother

\let\dim\relax
\DeclareMathOperator{\dim}{\mathsf{dim}}

\let\min\relax
\DeclareMathOperator{\min}{\mathsf{min}}

\let\max\relax
\DeclareMathOperator{\max}{\mathsf{max}}

\let\hom\relax
\DeclareMathOperator{\hom}{\mathsf{hom}}

\newcommand{\sym}{\mathfrak S}

\newcommand{\cA}{\mathcal A}

\newcommand{\cF}{\mathcal F}
\newcommand{\cG}{\mathcal G}

\newcommand{\cE}{\mathcal E}

\newcommand{\cZ}{\mathcal Z}

\newcommand{\cN}{\mathcal N}
\newcommand{\cC}{\mathcal C}

\newcommand{\cP}{\mathcal P}
\newcommand{\cQ}{\mathcal Q}

\newcommand{\cR}{\mathcal R}

\newcommand{\alt}{\mathfrak a}

\newcommand{\reg}{\mathcal O}

\renewcommand{\theta}{\vartheta}
\renewcommand{\phi}{\varphi}
\renewcommand{\_}{\underline{\,\,\,\,}}

\usepackage{lscape}
\usepackage{adjustbox}

\usepackage[usenames,dvipsnames]{xcolor}
\usepackage{hyperref}
\hypersetup{
    colorlinks,
    linkcolor={red!50!black},
    citecolor={blue!50!black},
    urlcolor={blue!80!black}
}
\usepackage{aliascnt} 

\usepackage{tikz-cd}

\newtheorem{theorem}{Theorem}[section]

\newaliascnt{conjecture}{theorem}

\aliascntresetthe{conjecture}

  \newaliascnt{proposition}{theorem}
  \newtheorem{prop}[proposition]{Proposition}
  \aliascntresetthe{proposition}

  \newaliascnt{lemma}{theorem}
  \newtheorem{lemma}[lemma]{Lemma}
  \aliascntresetthe{lemma}

  \newaliascnt{corollary}{theorem}
  
  \aliascntresetthe{corollary}

\theoremstyle{definition}

  \newaliascnt{definition}{theorem}
  
  \aliascntresetthe{definition}

  \newaliascnt{remark}{theorem}
  \newtheorem{remark}[remark]{Remark}
  \aliascntresetthe{remark}

  \newaliascnt{condition}{theorem}
  
  \aliascntresetthe{condition}

\newaliascnt{convention}{theorem}
 
\aliascntresetthe{convention}

\newaliascnt{assumption}{theorem}
 
\aliascntresetthe{assumption}

  \newaliascnt{question}{theorem}
  
  \aliascntresetthe{question}

  \newaliascnt{example}{theorem}
  \newtheorem{example}[example]{Example}
  \aliascntresetthe{example}

\newcommand{\olZ}{\overline Z}

\definecolor{viridislightgreen}{RGB}{94, 201, 98}
\definecolor{viridisdarkgreen}{RGB}{33, 145, 140}
\definecolor{viridisblue}{RGB}{59, 82, 139}
\definecolor{viridisviolet}{RGB}{68, 1, 84}

\makeatletter
\newcommand{\semiboxtimes}{%
  \mathbin{\mathpalette\semiboxtimes@{\relax}}%
}

\newcommand{\semiboxtimes@}[2]{%
  \begin{tikzpicture}[baseline=(B.base)]
    \node[inner sep=0pt,outer sep=0pt] (B)
      {$\m@th#1\phantom{\boxtimes}$};

    \begin{scope}
      \clip (B.south west) rectangle (B.north);
      \node[inner sep=0pt,outer sep=0pt] at (B.center)
        {$\m@th#1\otimes$};
    \end{scope}

    \begin{scope}
      \clip (B.south) rectangle (B.north east);
      \node[inner sep=0pt,outer sep=0pt] at (B.center)
        {$\m@th#1\boxtimes$};
    \end{scope}
  \end{tikzpicture}%
}
\makeatother

\tikzcdset{scale cd/.style={every label/.append style={scale=#1},
    cells={nodes={scale=#1}}}}

\begin{document}

 \title{Embeddings of Derived Categories of Moduli Spaces into Symmetric Powers}
 \author[A.\ Krug]{Andreas Krug}
\address{
Institut f\"ur algebraische Geometrie,
Gottfried Wilhelm Leibniz Universit\"at Hannover,
Welfengarten 1,
30167 Hannover,
Germany
}
\email{krug@math.uni-hannover.de}

  \maketitle
\begin{abstract}
We give a criterion for the box product of a Fourier--Mukai kernel to induce a fully faithful functor to the derived category of a symmetric quotient stack. This leads to many new examples of embeddings of derived categories of moduli spaces. We study in more detail the case of the moduli space of stable bundles on a curve. Concretely, we identify pieces of the left-orthogonal complement of our embedding, and extend the known semi-orthogonal decomposition of the moduli space of bundles with rank 2 and fixed determinant to the moduli space of bundles with varying determinant.
\end{abstract}

\section{Introduction}

Fourier--Mukai transforms play an important role in algebraic geometry; see \cite{Huy} for an overview. They are functors between the derived categories of varieties which are induced by objects in the derived category of the product. Of particular interest are fully faithful Fourier--Mukai transforms. Their images in the target category are admissible subcategories, which can be completed to semi-orthogonal decompositions; see \cite{Kuz--ICM} for an overview.

In this paper, we provide a method to construct fully faithful functors to
equivariant derived categories $\D_{\sym_n}(X^n)$ where $\sym_n$ is the symmetric group, acting on the cartesian power $X^n$ of a smooth projective variety by permutation of the factors. There is an isomorphism $\D_{\sym_n}(X^n)\cong \Sym^n(\D(X))$, where the right-hand side is the symmetric power of the category $\D(X)$ in the sense of Ganter--Kapranov \cite{GKsym}, which explains the title of the present paper.

The construction is quite simple, but gives a new perspective on some already known fully faithful embeddings and, more importantly, gives interesting new examples. In almost all our examples, the source category is the derived category of some moduli space of sheaves on $X$. Furthermore, the kernel of the induced Fourier--Mukai transform allows the same moduli space to also be interpreted as parametrising some class of $\sym_n$-equivariant sheaves on $X^n$.

Let us explain the results in more technical detail. Let $M$ and $X$ be smooth projective varieties over an algebraically closed field of characteristic zero $\kk$. We denote the bounded derived category of coherent sheaves by $\D(X)\coloneqq\D^b(\Coh(X))$.

We fix some object $\cU\in \D(M\times X)$ on the product.
For $t\in M$, we write $\iota_t\colon X\hookrightarrow M\times X$ for the embedding $\iota_t(x)=(t,x)$ and consider the fibre
\[
 \cU_t:=\iota_t^* \cU \in \D(X)\,.
\]
The pull-back $\iota_t^*$ is implicitly meant to be the derived functor, as is our general convention; see \autoref{subsect:conventions}.
However, in most of our examples, $\cU$ is a universal family, hence in particular flat over $M$, so that $\iota_t^*$ does not need to be derived.
Note that $\cU_t$ agrees with the image of the skyscraper sheaf $\reg_t$ under the Fourier--Mukai transform along $\cU$.

Let $\sym_n$ act on $M\times X^n$ via permutation of the factors of $X^n$. We define an induced object $\cU^{\lhd n}$ in the equivariant derived category $\D_{\sym_n}(M\times X^n)$ as follows. Let $\pr_i\colon X^n\to X$ be the projection to the $i$-th factor and $p_i=\id_M\times \pr_i\colon M\times X^n\to M\times X$. Then we set
\[
 \cU^{\lhd n}:=\cU^{\boxtimes_M n}:=\bigotimes_{i=1}^n p_i^*\cU\,.
\]
We have $\cU^{\lhd n}\cong (\delta\times \id_{X^n})^*\cU^{\boxtimes n}$ where $\delta\colon M\hookrightarrow M^n$ is the embedding of the small diagonal, and $\cU^{\boxtimes n}$ is the usual box product on $(M\times X)^n\cong M^n\times X^n$.
We denote the associated equivariant Fourier--Mukai transform by
\[I_n\coloneqq \FM_{\cU^{\lhd n}}\colon \D(M)\to \D_{\sym_n}(X^n)\,;\]
see \autoref{subsect:equi} and \autoref{subsect:box} for details.
For $t\in M$, let $i_t\colon X^n\cong \{t\}\times X^n\hookrightarrow M\times X^n$ be the induced embedding. Writing $(\cU^{\lhd n })_t\coloneqq i_t^*(\cU^{\lhd n})$, we get an isomorphism
\begin{equation}\label{eq:fibreiso}
(\cU^{\lhd n})_t\cong (\cU_t)^{\boxtimes n}\,.
\end{equation}
Hence, $I_n(\reg_t)\cong (\cU_t)^{\boxtimes n}\in \D_{\sym_n}(X^n)$ and $\cU^{\lhd n}$ can be regarded as the family of box powers of the fibres of $\cU$.
Let $d:=\dim M$. We prove the following criterion for full faithfulness of $I_n$.

\begin{theorem}\label{thm:mainintro}
If the graded Hom spaces between the fibers of $\cU$ satisfy the conditions
\begin{equation}\label{eq:maincond}
 \Hom^*(\cU_s,\cU_t)\cong \begin{cases}
                          \kk[0]\oplus \kk^{\oplus d}[-1]\quad&\text{if $s=t$,} \\
                          \kk^{\oplus d-1}[-1]\quad&\text{if $s\neq t$,}
                          \end{cases}
\end{equation}
then $I_n\colon \D(M)\to \D_{\sym_n}(X^n)$ is fully faithful for every $n\ge d$.
\end{theorem}

The proof is a quite straightforward application of the Bondal--Orlov criterion for testing full faithfulness on skyscraper sheaves (see \autoref{lem:BO} for the precise version that we are using) and the symmetric K\"unneth formula \autoref{lem:symKun}. Actually, \eqref{eq:maincond} is also necessary for $I_n$ to be fully faithful; see \autoref{thm:exactcriterion} for the precise statement. The proof of the if-and-only-if criterion is a bit more difficult than that of \autoref{thm:mainintro}. We postpone it to the last section, because, in the author's opinion, the most interesting aspects are the many new examples of fully faithful functors arising from \autoref{thm:mainintro}, which we discuss before in \autoref{sect:mostexamples} and \autoref{sect:M}.

First examples where the conditions of \autoref{thm:mainintro} are satisfied are:

\begin{itemize}
 \item If $M=\pt$ is a point and $\cU\in \D(\pt\times X)\cong \D(X)$ is an exceptional object. In this case $\cU^{\lhd n}=\cU^{\boxtimes n}\in \D_{\sym_n}(X^n)$ is again an exceptional object; see \autoref{subsect:exceptional}.
 \item If $M=C=X$ is a curve and $\cU=\reg_\Delta$. In this case $I_n\colon D(C)\hookrightarrow D_{\sym_n}(C^n)$ is the equivariant push-forward along the embedding of the small diagonal;  see \autoref{prop:diagonalff}.
\end{itemize}
As explained in \autoref{subsect:exceptional} and \autoref{subsect:curvediag}, the fully faithful functors arising from these first two examples are already well-known, but maybe recovering them via \autoref{thm:mainintro} gives an interesting new perspective on them.

Anyway, to the best of the author's knowledge, the fully faithful functors arising from the following examples are new to the literature. The conditions of \autoref{thm:mainintro} are also satisfied in the following cases:

\begin{itemize}
 \item If $X=\IP(W)$ is a projective space, $M=|\reg_X(m)|=\IP(S^m W^\vee)$ for $0<m<\dim W$, and $\cU$ is the universal family of degree $m$ hypersurfaces; see \autoref{ex:hypersurfaces}. Actually, this is a special case of a more abstract construction that generates a family $\cU$ satisfying the condition of \autoref{thm:mainintro} out of any strong exceptional pair; see \autoref{prop:pair}.
\item If $M$ is a smooth projective connected component of the moduli space of stable sheaves on a surface $X$ with universal family $\cU$, the condition of \autoref{thm:mainintro} is satisfied if and only if \[\Hom(\cU_s,\cU_t\otimes \omega_X)=0 \quad\text{for all $s,t\in M$;}\] see \autoref{prop:surfcondition}. This condition is always met
if $X$ is a del Pezzo surface, and if $X$ is an Enriques surface it is satisfied for moduli spaces of stable sheaves of odd rank with fixed determinant; see \autoref{subsect:surface}. Note that, for $X$ a surface, there is the derived McKay correspondence $\D_{\sym_n}(X^n)\cong \D(X^{[n]})$ where $X^{[n]}$ is the Hilbert scheme of points on $X$. Hence, in this case, we not only get an embedding of $\D(M)$ into the derived category of a stack, but also into the derived category of a variety.
 \item If $X=C$ is a curve, $M=M_C(r,e)$ with $r$ and $e$ coprime is the moduli space of stable bundles on $C$ of rank $r$ and degree $e$, and $\cU$ is the universal family; see \autoref{prop:modulicurves}.
 \end{itemize}
In \autoref{sect:M}, we further study the fully faithful functor $I_n\colon \D(M)\hookrightarrow \D_{\sym_n}(C^n)$ resulting from this last example and some related semi-orthogonal decompositions. By Polishchuk and van den Bergh \cite{PvdB}, there is a semi-orthogonal decomposition of $\D_{\sym_n}(C^n)$ whose pieces are indexed by the conjugacy classes of $\sym_n$ and which lift the canonical direct sum decomposition of orbifold cohomology. We prove that many of the pieces of this semi-orthogonal decomposition lie in the left-orthogonal complement of $I_n(\D(M))$; see \autoref{prop:complement}.

By \cite{TT} and \cite{Tevelev--nophantom}, there is a semiorthogonal decomposition of the derived category of the moduli space $M_C(2,\Lambda)$ of stable bundles of rank $2$ and fixed determinant $\Lambda$. We prove that there is an extension of this to a semi-orthogonal decomposition of $\D(M_C(2,e))$; see \autoref{thm:Msod}.
We then do some speculation on whether one might give an alternative construction of these semi-orthogonal decompositions of $\D(M_C(2,e))$ and $\D(M_C(2,\Lambda))$ using the fully faithful functor $I_n$ and the semiorthogonal decomposition of Polishchuk and van den Bergh of $\D_{\sym_n}(C^n)$, and possibly also obtain results regarding $\D(M_C(r,e))$ for higher $r$ in this way.

Here is how this paper is organised. In the preliminary \autoref{sect:prelim}, we fix our general notation, recall a few facts on equivariant sheaves and their derived categories, and introduce the relative box products, whose Fourier--Mukai transforms we study in this paper.
In the short \autoref{sect:sufficient}, we prove \autoref{thm:mainintro}.

In \autoref{sect:mostexamples}, we discuss the first four of the example classes listed above, one in every subsection. In \autoref{sect:M}, we discuss the case that $M$ is a moduli space of stable bundles on a curve. More precisely, in \autoref{subsect:Mholds}, we prove that the universal family fulfils the assumption of \autoref{thm:mainintro}, which gives the fully faithful embedding $I_n\colon \D(M)\hookrightarrow\D_{\sym_n}(C^n)$. In \autoref{subsect:PvdB}, we recall the semiorthogonal decomposition of $\D_{\sym_n}(C^n)$ constructed in \cite{PvdB} and prove that many of its pieces lie in the left-orthogonal complement of $I_n(\D(M))$. In \autoref{subsect:Lambdasod} we recall the semiorthogonal decomposition of \cite{TT} and \cite{Tevelev--nophantom} of the moduli space of bundles on $C$ with fixed determinant. \autoref{subsect:relativeFM} might be of independent interest. We first recall some basics on relative Fourier--Mukai transforms, and then provide a criterion \autoref{thm:relFMsod} which allows to test a relative semiorthogonal decomposition fibrewise. This might be well-known to experts, and there are closely related results in the literature, but we could not find the precise statement of \autoref{thm:relFMsod}(ii).
We then apply this criterion to extend the semiorthogonal decomposition of \cite{TT} and \cite{Tevelev--nophantom} from the moduli of bundles with fixed determinant to the whole $M$ in \autoref{subsect:extendsod}. Some speculation on connections of the various semiorthogonal decompositions is done in \autoref{subsect:possibleconnection}.

In the final \autoref{sect:necessary}, we strengthen \autoref{thm:mainintro} by showing that \eqref{eq:maincond} is not only sufficient, but also necessary for full faithfulness of $I_n=\FM_{\cU^{\lhd n}}$; see \autoref{thm:exactcriterion} for the precise statement.

\subsection*{Acknowledgements}

The studies carried out in this paper were inspired by the $\ell=1$ case of
\autoref{ex:hypersurfaces}. More concretely, this is the case where $X=\IP^m$ is a projective space, $M=(\IP^m)^*$ is the dual projective space, and $\cU=\reg_\cH$ is the structure sheaf of the universal hyperplane section. In this case, the full faithfulness of $I_n\colon \D((\IP^m)^*)\to \D_{\sym_n}((\IP^m)^n)$ for $n\ge m$ was already found some time ago (probably in 2013) by the author together with Evgeny Shinder.

\subsection*{AI Disclosure}

Most results of this paper were found in ``discussions'' with LLMs. The writing was done by the author without assistance, with an LLM only used for proofreading.

\section{Preliminaries}\label{sect:prelim}

\subsection{General Conventions}\label{subsect:conventions}

Throughout this paper, $\kk$ is an algebraically closed field of characteristic zero. We only drop the assumption on the characteristic of $\kk$ in \autoref{subsect:relativeFM}. It is not necessary there, and the criterion might be of independent interest, so we decided to state it in the bigger generality.

Throughout, $X$ and $M$ will be smooth projective varieties over $\kk$.
In this paper, with the exception of \autoref{lem:componentvanish}, varieties are assumed to be connected. If we write $x\in X$, we mean that $x$ is a closed point of $X$. We write $\pt:=\Spec \kk$.

By the \emph{derived category of $X$}, we mean the bounded derived category of coherent sheaves $\D(X):=\D^b(\Coh(X))$.
For two objects $A,B\in \D(X)$, we consider the graded vector space
\begin{equation}\label{eq:gradedHom}
\Hom^*(A,B):=\bigoplus_{i\in \IZ} \Hom_{\D(X)}\bigl(A, B[i]\bigr)[-i]\,.
\end{equation}
If $A, B$ are sheaves concentrated in degree zero, we have $\Hom^*(A,B)\cong\Ext^*(A,B)$.

All geometric functors occurring in this paper are implicitly meant to be derived, though we will not make this explicit in the notation. For example, given a morphism $f\colon X\to Y$ between smooth projective varieties, when we write $f^*$, we mean the derived pull-back $Lf^*\colon \D(Y)\to \D(X)$, and when we write $f_*$, we mean the derived push-forward $Rf_*\colon \D(X)\to \D(Y)$. The sole exception is \autoref{lem:componentvanish} where a restriction to a subvariety is taken in the underived sense.

\subsection{Equivariant Sheaves}\label{subsect:equi}

For details on equivariant sheaves and their derived categories, we refer to \cite{BOber--equi}, \cite[Sect.\ 2.2]{Krug--remarksMcKay}, \cite[Sect.\ 4]{BKR}. In this subsection, we only collect a few facts needed later.

Let $G$ be a finite group acting on a variety $Y$. We denote the abelian category of $G$-equivariant coherent sheaves on $Y$ by $\Coh_G(Y)$.
Its objects are pairs $(E,\lambda)$ with $E\in \Coh(Y)$ and $\lambda$ a \emph{$G$-linearisation}, i.e.\ a family of isomorphisms $\lambda=(\lambda_g\colon E\xrightarrow\cong g^*E)_{g\in G}$ satisfying
\[
\begin{tikzcd}
E \arrow[rrr, bend right=18, "\lambda_{hg}" ]      \arrow[r, , "\lambda_{g}"]         & g^*E    \arrow[r, "g^*\lambda_{h}"]           & g^*h^*E \arrow[r, "\cong"] & (hg)^*E \,.  \\
\end{tikzcd}
\]
We will often omit $\lambda$ in the notation and simply write $E$ for the equivariant sheaf.
For two $G$-equivariant sheaves $E$ and $F$, there is an induced $G$-action on $\Hom(E,F)$, given by conjugation by the $G$-linearisations.
The morphisms in $\Coh_G(Y)$ are given by
\[
 \Hom_G(E,F):=\Hom(E,F)^G\,.
\]
The bounded derived category of $\Coh_G(Y)$ is denoted by $\D_G(Y):=\D^b\bigl(\Coh_G(Y)\bigr)$. Similarly to our notation \eqref{eq:gradedHom} for non-equivariant derived categories, we write the graded Hom-space between two objects $A, B\in \D_G(Y)$ by
\[
 \Hom_G^*(A,B):=\bigoplus_{i\in \IZ} \Hom_{\D_G(Y)}\bigl(A, B[i]\bigr)[-i]\,.
\]
The $G$-linearisations of $A$ and $B$ induce a $G$-action on the graded-Hom space $\Hom^*(A,B)$ between the underlying objects in $\D(Y)$, and $\Hom^*_G(A,B)$ consists of the invariants under this action:
\begin{equation}\label{eq:HomG}
\Hom_G^*(A,B)\cong\bigl( \Hom^*(A,B)\bigr)^G\,.
\end{equation}
If $G$  acts trivially on $Y$, a linearisation of a coherent sheaf $E\in \Coh(Y)$ is the same as a $G$-action, i.e.\ a group homomorphism $G\to \Aut(E)$. In this case, there is the functor $\triv\colon \Coh(Y)\to \Coh_G(Y)$ which equips every sheaf with the trivial $G$-action. It has the functor $(\_)^G\colon \Coh_G(Y)\to \Coh(Y)$, which sends any sheaf $F$ equipped with a $G$-action to the sheaf of $G$-invariants $F^G$, as its both-sided adjoint.

Given a $G$-equivariant morphism $f\colon X\to Y$ between two smooth projective varieties with $G$-actions, there is an equivariant push-forward $f_*\colon \Coh_G(X)\to \Coh_G(Y)$ and an equivariant pull-back $f^*\colon \Coh_G(Y)\to \Coh_G(X)$, which are compatible with the non-equivariant push-forward and pull-back, respectively. There is also a tensor product of equivariant sheaves, and all these functors can be derived to give functors on the level of the equivariant derived categories.

Accordingly, there is also the notion of an equivariant Fourier--Mukai transform; see e.g.\ \cite[Sect.\ 1]{Plo}, \cite[Sect.\ 2.1]{Krug--NakaP}. We will only need the special case where one of the two varieties is equipped with a trivial group action. So, let $X$ and $Y$ be smooth projective varieties, and let a finite group $G$ act on $Y$. Then, there is an induced $G$ action on $X\times Y$, given by $g(x,y)=(x,gy)$. For $\cP\in \D_G(X \times Y)$, the induced \emph{equivariant Fourier--Mukai transform} $\FM_{\cP}\colon \D(X)\to \D_G(Y)$ is the composition
\[
 \D(X)\xrightarrow{\triv} \D_G(X) \xrightarrow{\pr_{X}^*} \D_G(X\times Y)\xrightarrow{(\_)\otimes \cP} \D_G(X\times Y) \xrightarrow{\pr_{Y*}} \D_G(Y)\,.
\]
We will need the following in the proof of \autoref{prop:complement}.

\begin{lemma}\label{lem:componentvanish}
 Let $Z$ be a, not necessarily connected, variety on which a finite group $G$ acts. Let $Y\subset Z$ be a subvariety with $G\cdot Y=Z$, and let $f\colon Z\to B$ be an affine $G$-invariant morphism. Let $W\subset G$ be a subgroup which acts trivially on $Y$, which means $w\cdot y=y$ for every $w\in W$ and $y\in Y$. Then, for every torsion-free $G$-equivariant sheaf $E\in \Coh_G(Z)$ we have
 \[
(E_{\mid Y})^W\cong 0\quad\implies \quad (f_*E)^G\cong 0\,.
 \]
\end{lemma}

Here, for once, the restriction $(\_)_{\mid Y}$ denotes the underived pull-back to $Y$.

\begin{proof}
 Since $f$ is affine, the asserted vanishing $(f_*E)^G\cong 0$ is equivalent to $\Gamma(f^{-1}U, E)^G=0$ for every affine open $U\subset B$. This vanishing, under the given assumptions, is \cite[Lem.\ 2.1]{Kru--Gm}.
\end{proof}

\subsection{Box Products}\label{subsect:box}
Given a smooth projective variety $X$, there is a canonical action of the symmetric group $\sym_n$ on $X^n$ by permutation of the factors.

Let $\pr_i\colon X^n\to X$ be the projection to the $i$-th factor. For $A\in \D(X)$, the box product
\[
 A^{\boxtimes n}:=\pr_1^*A\otimes \dots \otimes \pr_n^*A
\]
carries a canonical $\sym_n$-linearisation given by permutation of the factors of the tensor product. When we write $A^{\boxtimes n}\in \D_{\sym_n}(X^n)$, we mean the box product equipped with this canonical linearisation.

From the non-equivariant K\"unneth isomorphism together with \eqref{eq:HomG}, we get

\begin{lemma}[Symmetric K\"unneth formula]\label{lem:symKun}
For every $A,B\in \D(X)$, we have
\[
\Hom_{\sym_n}^*(A^{\boxtimes n}, B^{\boxtimes n})\cong S^n\Hom^*(A,B)\,.
 \]
\end{lemma}

\begin{remark}\label{rem:graded}
 The symmetric product of $V=\Hom^*(A,B)$ is taken in the graded sense. This means that, if we decompose
$V=V^{\mathsf{even}}\oplus V^{\mathsf{odd}}$ into graded vector spaces concentrated in even and odd degrees, we have
\[
 S^n V\cong \bigoplus_{i=0}^n S^{n-i} V^{\mathsf{even}}\otimes S^{i} V^{\mathsf{odd}}
\]
where the underlying ungraded vector space of $S^{n-i} V^{\mathsf{even}}$ is really the symmetric power while the underlying ungraded vector space of
$S^{i} V^{\mathsf{odd}}$ is the $i$-th wedge power of the underlying ungraded vector space of $V^{\mathsf{odd}}$. In other words,
\[
 S^{i} V^{\mathsf{odd}}\cong \bigr(\wedge^i(V^{\mathsf{odd}}[1])\bigl)[-i]\,.
\]
\end{remark}
Let now $\cU\in \D(M\times X)$ and let us write $p_i:=(\id_M\times \pr_i)\colon M\times X^n\to M\times X$. Then, the \emph{relative box product over $M$}
\[
 \cU^{\boxtimes_M n}:=p_1^*\cU\otimes \dots\otimes p_n^*\cU
\]
has again a canonical linearisation, and we write the object equipped with this linearisation as $\cU^{\lhd n}:= \cU^{\boxtimes_M n}\in \D_{\sym_n}(M\times X^n)$. The induced equivariant Fourier--Mukai transform is written
\[
I_n:=\FM_{\cU^{\lhd n}}\colon \D(M)\to \D_{\sym_n}(X^n)\,.
\]

\begin{remark}\label{rem:Idelta}
Let $\delta_n\colon M\hookrightarrow M^n$ be the embedding of the small diagonal and consider the induced embedding
\[
 \delta_n\times \id_{X^n}\colon M\times X^n\hookrightarrow M^n\times X^n\cong (M\times X)^n\,.
\]
Then $\cU^{\lhd n}\cong (\delta_n\times \id_{X^n})^*\cU^{\boxtimes n}$.
By the equivariant analogue \cite[Lem.\ 2.7]{KSequi} of \cite[Prop.\ 2.1.6]{Orlsurvey}, we get the isomorphism of functors
\[I_n\cong \FM_{\cU^{\boxtimes n}}\circ \delta_{n*}\circ \triv\,.\]
\end{remark}

\subsection{Continuity of Fibres}

The following is surely well-known, but we could not find a precise reference for objects in the derived category instead of for sheaves.

\begin{lemma}\label{lem:Eulerconst}
Let $Y$ be a smooth projective variety and $\cA\in \D(Y)$.
For a closed point $y\in Y$, write $i_y\colon \pt\hookrightarrow Y$ for its embedding.
\begin{enumerate}
 \item There exist integers $a\le b$ such that $\Ho^i(i_y^*\cA)=0$ for every $i\notin[a,b]$ and every $y\in Y$.
\item\label{Euler} The Euler characteristic $\chi(i_y^*\cA)$ is the same for every closed point $y\in Y$.
 \end{enumerate}
\end{lemma}

Recall that $i_y^*$ denotes the derived pull-back. This is necessary for \ref{Euler} to hold without any flatness assumption on $\cA$.

\begin{proof}
Every object in $\D(Y)=\D^b(\Coh(Y))$ is isomorphic to a perfect complex; see e.g.\ \cite[Sect.\ 3.1]{BvdB}. Hence, locally $\cA$ is isomorphic to a bounded complex of vector bundles. Let us fix some non-empty open subset $W\subset Y$ and some complex
\[
 0\to E^a\to\dots\to E^b\to 0
\]
of vector bundles on $W$ which is quasi-isomorphic to $\cA_{\mid W}$.
Then, for every $y\in W$, we have
\[
 i_y^*\cA\cong \bigl(0\to E^a(y)\to\dots\to E^b(y)\to 0  \bigr)
\]
This implies the statement of (i) for every $y\in W$. The global statement follows by quasi-compactness.

Furthermore, $\chi(i_y^*\cA)=\sum_{i=a}^b(-1)^i\rank(E^i)$ is independent of $y\in W$. By connectedness of $Y$, we get that $\chi(i_y^*\cA)$ is the same for every $y\in Y$.
\end{proof}

\section{Proof of \autoref{thm:mainintro}}\label{sect:sufficient}

In this section, we prove \autoref{thm:mainintro} as an easy application of the stacky version \cite{Lim-Poli--BO} of the Bondal--Orlov criterion \cite{BO--ffcrit}.

Let $\cU\in \D(M\times X)$ be an object satisfying for $s,t\in M$
\begin{equation}\label{eq:proofcond}
 \Hom^*(\cU_s,\cU_t)\cong\begin{cases}  \kk[0]\oplus \kk^{\oplus d}[-1]\quad&\text{if $s=t$,} \\
                          \kk^{\oplus d-1}[-1]\quad&\text{if $s\neq t$.}
\end{cases}
                          \end{equation}
In order to lighten the notation, we set $\cW:=\cU^{\lhd n}$ which turns \eqref{eq:fibreiso} into
\begin{equation}\label{eq:boxfibre}
 \cW_t\cong (\cU_t)^{\boxtimes n}
\end{equation}
as objects in $\D_{\sym_n}(X^n)$ for every $t\in M$. With $I_n:=\FM_{\cW}$, hence $I_n(\reg_t)\cong \cW_t$, the criterion
\cite[Thm.\ 1.1]{Lim-Poli--BO} can be phrased as
\begin{lemma}\label{lem:BO}
$I_n$ is fully faithful if and only if the following conditions all hold:
\begin{enumerate}
 \item\label{cond:selfHom} For every $t\in M$ we have $\Hom^0_{\sym_n}(\cW_t, \cW_t)\cong \kk$ and $\Hom^i_{\sym_n}(\cW_t, \cW_t)\cong 0$ for all $i\notin[0,d]$.
 \item\label{cond:ortho} For every $s,t\in M$ with $s\neq t$, we have $\Hom^*_{\sym_n}(\cW_s, \cW_t)\cong 0$.
\end{enumerate}
\end{lemma}

So we only need to show that, for $n\ge d$, \eqref{eq:proofcond} implies conditions \ref{cond:selfHom} and \ref{cond:ortho} of \autoref{lem:BO}, which we now do.

The condition $\Hom^*(\cU_t,\cU_t)\cong\kk[0]\oplus\kk^{\oplus d}[-1]$ together with \eqref{eq:boxfibre} and the symmetric K\"unneth formula \autoref{lem:symKun} gives
\[
 \Hom^*_{\sym_n}(\cW_t,\cW_t)\cong S^n\bigl(\kk[0]\oplus \kk^{\oplus d}[-1]\bigr)\cong  \bigoplus_{i=0}^n \wedge^i(\kk^{\oplus d})[-i]\cong \bigoplus_{i=0}^d \wedge^i(\kk^{\oplus d})[-i]
\]
where the last isomorphism uses $n\ge d$. Hence, \ref{cond:selfHom} of \autoref{lem:BO} is satisfied.
Similarly, the condition $\Hom^*(\cU_s,\cU_t)\cong \kk^{\oplus d-1}[-1]$ for $s\neq t$ together with \eqref{eq:boxfibre} and \autoref{lem:symKun} gives
\[
\Hom^*_{\sym_n}(\cW_s,\cW_t)\cong\wedge^n(\kk^{\oplus d-1})[-n]\cong 0
\]
where the vanishing is because of the condition $n\ge d$. This shows that also \ref{cond:ortho} of \autoref{lem:BO} is satisfied, which finishes the proof of \autoref{thm:mainintro}.

\section{Examples}\label{sect:mostexamples}

\subsection{Exceptional Objects}\label{subsect:exceptional}

In the case that $M=\pt$ is a point, we have the canonical identification $\pt\times X\cong X$. Since $d=\dim(\pt)=0$, an object $\cU\in \D(\pt\times X)\cong \D(X)$ satisfies the conditions of \autoref{thm:mainintro} if and only if it is exceptional, which means that $\Hom^*_X(\cU,\cU)=\kk\cdot \id[0]$. Under the isomorphism $\pt\times X^n\cong X^n$, we have $\cU^{\lhd n}\cong \cU^{\boxtimes n}$. That the induced Fourier--Mukai transform
\[
 I_n:=\FM_{\cU^{\lhd n}}\colon \D(\pt)\to \D_{\sym_n}(X^n)
\]
is fully faithful then just means that $\cU^{\boxtimes n}\in \D_{\sym_n}(X^n)\cong \D_{\sym_n}(\pt\times X^n)$ is again an exceptional object, which was already well-known; see e.g.\ \cite[Sect.\ 4]{KSos}.

\subsection{The Diagonal Embedding for Curves}\label{subsect:curvediag}

Let now $C$ be a smooth projective curve, and $M=C=X$. Hence, $d=\dim M=1$. We consider $\cU=\reg_\Delta\in \D(C\times C)$ the structure sheaf of the diagonal. Then, for every $x\in C$, we have $\cU_x\cong \reg_x$. Hence, for $s,t\in C$, we have
\[
 \Hom^*(\cU_s,\cU_t)\cong \begin{cases}
                          \kk[0]\oplus \kk[-1]=\kk[0]\oplus \kk^{\oplus d}[-1]\quad&\text{if $s=t$,} \\
                          0=\kk^{\oplus d-1}[-1]\quad&\text{if $s\neq t$.}
                          \end{cases}
\]
Hence, the conditions of \autoref{thm:mainintro} are fulfilled. Together with \autoref{rem:Idelta}, this gives
\begin{prop}\label{prop:diagonalff}
 For every $n\ge 1$, the equivariant push-forward along the embedding of the small diagonal
 $\delta_{n*}\circ \triv \colon \D(C)\hookrightarrow \D_{\sym_n}(C^n)$ is fully faithful.
\end{prop}

\begin{remark}\label{rem:lambdan}
That $\delta_{n*}\circ \triv$ is fully faithful was already known. Its image is part of a semi-orthogonal decomposition of $\D_{\sym_n}(C^n)$ with more pieces; see
\cite[Thm.\ 1.1]{Krug--NakaP} or \cite[Thm.\ B]{PvdB}. Note, however, that the known proofs differ quite a lot from the one going via \autoref{thm:mainintro}. The one in \cite{Krug--NakaP} uses the equivariant normal bundle of the diagonal and the derived self-intersection formula, while the one in \cite{PvdB} uses the Springer correspondence (to prove full faithfulness of a much bigger collection of functors; see \autoref{subsect:PvdB} for some details).
\end{remark}

More generally, let $C$ be a curve and $X$ of arbitrary dimension, and let $\cU\in \D(C\times X)$ be a Fourier--Mukai kernel of any fully faithful functor. Then $\Hom^*(\cU_s,\cU_t)\cong\Hom^*(\reg_s,\reg_t)$. Hence, the condition of \autoref{thm:mainintro} is met and $\FM_{\cU^{\lhd n}}\colon \D(C)\to \D_{\sym_n}(X^n)$ is again fully faithful. However, this gives nothing new, as by \autoref{rem:Idelta}, we have
\[
 \FM_{\cU^{\lhd n}}\cong \FM_{\cU^{\boxtimes n}}\circ \delta_{n*}\circ \triv\,.
\]
The first part $\delta_{n*}\circ\triv$ of this composition is already fully faithful by \autoref{prop:diagonalff}, and also the second part $\FM_{\cU^{\boxtimes n}}\colon \D_{\sym_n}(C^n)\to \D_{\sym_n}(X^n)$ is well-known to be fully faithful; compare \cite[Sect.\ 3.1]{Plo}, \cite[Rem.\ 4.7]{KSos}, \cite{Koseki--sod}.

\subsection{Examples from Strong Exceptional Pairs}

Let $(E_0,E_1)$ be a strong exceptional pair in $\D(X)$ for some smooth projective variety $X$.
This means that $E_0,E_1\in \D(X)$ are exceptional objects, $\Hom^*(E_1,E_0)=0$ and $\Hom^*(E_0,E_1)=\Hom(E_0,E_1)[0]$ is concentrated in degree zero.

Assume from now on that $V\coloneqq\Hom(E_0,E_1)\neq 0$ so that the pair is not fully orthogonal. Also set $M\coloneqq \IP(V^\vee)$ which parametrises lines in $V$. We define on $M\times X$ the object
\[
 \cU:=\cone\bigl( \reg_M(-1)\boxtimes E_0\xrightarrow{\mathfrak t} \reg_M\boxtimes E_1\bigr)
\]
where $\mathfrak t$ is the morphism corresponding to $\id_V$ under the canonical isomorphisms
\[
\Hom\bigl( \reg_M(-1)\boxtimes E_0, \reg_M\boxtimes E_1\bigr)\cong \Hom(\reg_M(-1), \reg_M)\otimes \Hom(E_0,E_1)\cong V^\vee\otimes V\,.
\]
Identifying a point $t\in M$ with a line $\langle t\rangle \subset V=\Hom(E_0,E_1)$, the fibre of $\cU$ is given by
\[
 \cU_t\cong \cone\bigl( E_0\xrightarrow t E_1\bigr)\,.
\]

\begin{prop}\label{prop:pair}
$\cU$ satisfies condition \eqref{eq:maincond} of \autoref{thm:mainintro} with $d=\dim(M)=\dim V-1$.
\end{prop}

\begin{proof}
 The exact triangle $E_0\xrightarrow t E_1\to \cU_t\to$ together with its analogue for $s$ in place of $t$ induces the commutative diagram
\begin{equation}\label{eq:Homdiag}\begin{tikzcd}
	{\Hom^*(\mathcal U_s,E_0)} & {\Hom^*(\mathcal U_s,E_1)} & {\Hom^*(\mathcal U_s,\mathcal U_t)} & {}  \\
	{\Hom^*(E_1,E_0)} & {\Hom^*(E_1,E_1)} & {\Hom^*(E_1,\mathcal U_t)} & {} \\
	{\Hom^*(E_0,E_0)} & {\Hom^*(E_0,E_1)} & {\Hom^*(E_0,\mathcal U_t)} & {} \\
{}	 & {} & {} &{}
	\arrow["t_*", from=1-1, to=1-2]
	\arrow[from=1-1, to=2-1]
	\arrow[from=1-2, to=1-3]
	\arrow[from=1-2, to=2-2]
	\arrow[from=1-3, to=1-4]
	\arrow[from=1-3, to=2-3]
	\arrow["t_*",from=2-1, to=2-2]
	\arrow["s^*",from=2-1, to=3-1]
	\arrow[from=2-2, to=2-3]
	\arrow["s^*",from=2-2, to=3-2]
	\arrow[from=2-3, to=2-4]
	\arrow["s^*",from=2-3, to=3-3]
	\arrow["t_*",from=3-1, to=3-2]
	\arrow[from=3-1, to=4-1]
	\arrow[from=3-2, to=3-3]
	\arrow[from=3-2, to=4-2]
	\arrow[from=3-3, to=3-4]
	\arrow[from=3-3, to=4-3]
\end{tikzcd}\end{equation}
where all columns and rows are exact triangles in the derived category of vector spaces corresponding to long exact sequences of the Hom spaces in the individual degrees.
We have $\Hom^*(E_0,E_0)=\kk\cdot \id_{E_0}$. Hence, from the lower line, we get $\Hom^*(E_0,\mathcal \cU_t)\cong V/\langle t\rangle$, concentrated in degree $0$.
Since $\Hom^*(E_1,E_0)=0$, the middle line gives \[\kk\cdot \id_{E_1} = \Hom^*(E_1,E_1)\cong \Hom^*(E_1,\cU_t)\,.\] Hence, the right column is isomorphic to
\begin{equation}\label{eq:sesst}
 \Hom^*(\cU_s,\cU_t)\to \kk \xrightarrow{1\mapsto s} V/\langle t\rangle\to \Hom^*(\cU_s,\cU_t)[1] \,.
\end{equation}
If $s$ and $t$ define the same point of $\IP(V^\vee)$, the middle arrow is zero, hence
\[
\Hom^*(\cU_t,\cU_t)\cong \kk[0]\oplus V/\langle t\rangle [-1]\cong \kk[0]\oplus \kk^d[-1] \,.
\]
If $s$ and $t$ define different points of $\IP(V^\vee)$, then \eqref{eq:sesst} gives
\[
\Hom^*(\cU_s,\cU_t)\cong V/\langle s, t\rangle [-1]\cong \kk^{d-1}[-1] \,.\qedhere
\]
 \end{proof}

\begin{example}\label{ex:hypersurfaces}
Let $X=\IP(W)$ for some vector space $W$, and let $\ell$ be an integer with $0<\ell<\dim(W)$. Then $E_0=\reg_X(-\ell)$ and $E_1=\reg_X$ form a strong exceptional sequence. Then $V=\Hom(E_0,E_1)\cong S^\ell W$ and the universal cone $\cU$ on $M\times X=\IP(S^\ell W^\vee)\times \IP(W)$ agrees with the structure sheaf of the universal family of degree $\ell$ hypersurfaces. Hence, for $n\ge d=\binom{\dim W+\ell -1}{\dim W-1}-1$, we get a fully faithful functor $I_n\colon \D(\IP(S^\ell W^\vee))\hookrightarrow\D_{\sym_n}(\IP(W)^n)$ whose Fourier--Mukai kernel $\cU^{\lhd n}$ is the structure sheaf of the $\sym_n$-invariant subvariety
\[
 \cZ_n=\Bigl\{(D,x_1,\dots,x_n)\mid D\subset X \text{ is a degree $\ell$ hypersurface, }x_1,\dots,x_n\in D\Bigr\}\subset M\times X^n \,.
\]
For $\dim W=3$, where $X\cong \IP^2$ is the projective plane, $\ell=1$, and $n=2$,
the fully faithful functor $I_2\colon \D(\IP(W^\vee))\to \D_{\sym_2}(\IP(W)^2)$
composed with the derived McKay correspondence $\D_{\sym_2}(\IP(W)^2)\cong \D(\IP(W)^{[2]})$ agrees with the pull-back along the $\IP^2$-fibration $\IP(W)^{[2]}\to \IP(W^\vee)$ which is given by sending a length 2 cluster to the unique line it is contained in; compare the proof of \cite[Prop.\ 6.6]{Hochenegger-Krug--Pformal}.
\end{example}

\subsection{Moduli spaces of stable sheaves on surfaces}\label{subsect:surface}

\begin{prop}\label{prop:surfcondition}
 Let $X$ be a smooth projective surface, and let $M$ be a smooth projective component of the moduli space of Gieseker stable sheaves on $X$ with respect to some fixed polarisation, which has a universal family $\cU\in \Coh(M\times X)$. Then, $\cU$ satisfies condition \eqref{eq:maincond} of \autoref{thm:mainintro} if and only if
 \begin{equation}\label{eq:omegacond}
  \Hom(\cU_t,\cU_s\otimes \omega_X)=0
 \end{equation}
for all $s,t\in M$ (including the case $s=t$).
\end{prop}

\begin{proof}
By Serre duality, $\Hom(\cU_t,\cU_s\otimes \omega_X)\cong \Ext^2(\cU_s, \cU_t)^\vee$, which shows directly that \eqref{eq:omegacond} is necessary.
Let us now assume that \eqref{eq:omegacond} holds, and deduce \eqref{eq:maincond} from it.
For stable sheaves of the same Hilbert polynomial, we have
\begin{equation}\label{eq:Homstable}
 \Hom(\cU_s,\cU_t)=\begin{cases}
                    \kk\cdot \id\quad&\text{if $s=t$,}\\
                   0\quad&\text{if $s\neq t$;}
                   \end{cases}
\end{equation}
see e.g.\ \cite[Prop.\ 1.2.7]{HL}.
Since in the $s=t$ case, $\Ext^1(\cU_t,\cU_t)\cong T_M(t)\cong \kk^d$ by smoothness of $M$, we also get immediately that \eqref{eq:omegacond} implies $\Hom^*(\cU_t,\cU_t)\cong \kk[0]\oplus \kk^{\oplus d}[-1]$.

Let $\tau\colon X\times M\to M\times X$ be the morphism which switches the factors, and let $\pr_{ij}$ denote the projection from the triple product $M\times X\times M$ to pairs of factors. We consider
\[
 \cA:=\pr_{13*}\sHom(\pr_{12}^* \cU, \pr_{23}^*\tau^* \cU)\,\in \D(M\times M)\,.
\]
Then, if we denote the embedding of a point $(s,t)\in M\times M$ by $\iota_{s,t}\colon \pt\hookrightarrow M\times M$, we have
\begin{equation}\label{eq:Afibres}
 \iota_{s,t}^*\cA\cong \Hom^*(\cU_s,\cU_t)
\end{equation}
by base change and compatibility of $\sHom$ with pull-backs (recall our convention that all functors are derived).
For $s=t$, we already know that $\chi\bigl(\Hom^*(\cU_t,\cU_t)\bigr)=1-d$. Hence, by \autoref{lem:Eulerconst}, also
$\chi\bigl(\Hom^*(\cU_s,\cU_t)\bigr)=1-d$ for $s\neq t$.
Together with the vanishings \eqref{eq:omegacond} and \eqref{eq:Homstable}, this gives $\Hom^*(\cU_s,\cU_t)\cong \kk^{d-1}[-1]$.
\end{proof}


\begin{remark}
 Instead of using the object $\cA:=\pr_{13*}\sHom(\pr_{12}^* \cU, \pr_{23}^*\tau^* \cU)\,\in \D(M\times M)$, we can alternatively obtain the constancy of the Euler characteristics from Riemann--Roch because the numerical class of the Chern classes of the $\cU_t$ does not vary in the family. However, we need a similar argument involving an analogous object $\cA$ later in the proof of \autoref{thm:exactcriterion}, which is why we decided to present the argument in this form already here.
\end{remark}

Let us mention two classes of moduli spaces where \eqref{eq:omegacond} is fulfilled. The first is when $X$ is a del Pezzo surface. Then every moduli space $M$ satisfies the condition (as long as it parametrises stable sheaves with positive dimensional support, which rules out $X$ seen as a moduli space of skyscraper sheaves on $X$). Indeed, we have the inequality $p(\cU_t)>p(\cU_s\otimes \omega_X)$ of the reduced Hilbert polynomials, which gives the vanishing $\Hom(\cU_t,\cU_s\otimes \omega_X)=0$. Indeed, if we assume that there is a non-zero homomorphism $\cU_t\to \cU_s\otimes \omega_X$ with image $Q\subset \cU_s\otimes \omega_X$, stability of $\cU_t$ and $\cU_s$ together with ampleness of $\omega_X^{-1}$ give the contradiction
\[
 p(\cU_t)\le p(Q) < p(Q\otimes\omega_X^{-1})\le p(\cU_s)=p(\cU_t)\,.
\]
Let now $X$ be an Enriques surface. As $q(X)=0$, fixing the determinant in addition to the Chern classes gives full components of the moduli of semi-stable sheaves on $X$. Let us fix such a component $M$ which is smooth and parametrises stable sheaves of odd rank. Then, for $s,t\in M$, the sheaves $\cU_t$ and $\cU_s\otimes \omega_X$ are non-isomorphic because their determinants differ by $\omega_X$. However, they still have the same Hilbert polynomial, hence
$\Hom(\cU_t,\cU_s\otimes \omega_X)=0$ again by \cite[Prop.\ 1.2.7]{HL}.

\begin{remark}
For $X$ a smooth projective surface, there is an equivalence $\D_{\sym_n}(X^n)\cong \D(X^{[n]})$ which is a special case of the derived McKay correspondence; see \cite{BKR} and \cite{Hai}. Hence, the above gives in many cases a fully faithful functor from moduli spaces of sheaves on $X$ to a Hilbert scheme of points on $X$.

In the case that $M$ parametrises only vector bundles, after making the correct concrete choice of the equivalence $\Psi\colon \D_{\sym_n}(X^n)\to \D(X^{[n]})$ and denoting its base change by
\[\Psi_M\colon \D_{\sym_n}(M\times X^n)\to \D(M\times X^{[n]})\] one gets that $\cF:=\Psi_M(\cU^{\lhd n})$ is a vector bundle on $M\times X^{[n]}$; see \cite[Lem.\ 1.1]{KRZ--gentaut}.
Furthermore, it is shown in \cite{KRZ--gentaut} under the additional assumption that all fibres of $\cU$ are slope stable
that $\cF$ is a universal family of slope stable bundles on $X^{[n]}$ with respect to some polarisation, which turns $M$ into a complete connected component of the moduli space of stable sheaves on $X^{[n]}$. Indeed, this is the special case $\lambda=(n)$ of \cite[Thm.\ 0.1 \& 0.2]{KRZ--gentaut}.

We have $\FM_{\cF}\cong \Psi\circ \FM_{\cU^{\lhd n}}\cong \Psi\circ I_n$; see e.g.\ \cite[Prop.\ 2.1.6]{Orlsurvey}. Hence, in the cases described above, we start with a moduli space of stable bundles on a surface $X$, identify it with a moduli space of stable bundles on a Hilbert scheme, and the universal family of this moduli space induces a fully faithful functor to the derived category of the Hilbert scheme.
\end{remark}

\section{Moduli of Stable Bundles on Curves}\label{sect:M}

\subsection{Establishing the Fully Faithful Embedding}\label{subsect:Mholds}

Let $C$ be a smooth projective curve of genus $g=g(C)\ge 1$. Let $r,e$ be coprime integers with $r\ge 1$. Then the moduli space $M\coloneqq M_C(r,e)$ of stable vector bundles of rank $r$ and degree $e$ on $C$ is smooth and projective of dimension
\[
 d:=\dim(M)=r^2(g-1)+1\,;
\]
see e.g.\ \cite[Thm.\ 3.7(2)]{MR4773901}.
Let $\cU\in \D(M\times C)$ denote the universal family.

\begin{prop}\label{prop:modulicurves}
For $s,t\in M$, we have
\[
 \Hom^*(\cU_s,\cU_t)\cong \begin{cases}
                          \kk[0]\oplus \kk^{\oplus d}[-1]\quad&\text{if $s=t$,} \\
                          \kk^{\oplus d-1}[-1]\quad&\text{if $s\neq t$.}
                          \end{cases}
\]
Hence, the assumptions of \autoref{thm:mainintro} are satisfied, which gives that the functor
\[
 I_n=\FM_{\cU^{\lhd n}}\colon \D(M)\to \D_{\sym_n}(C^n)
\]
is fully faithful for every $n\ge d$.
\end{prop}

\begin{proof}
 The proof is essentially the same as that of \autoref{prop:surfcondition}, only simpler because $\Ext^2$ vanishes automatically. Indeed, by stability of the bundles, we have
\[
 \Hom(\cU_s,\cU_t)\cong \begin{cases}
                          \kk \quad&\text{if $s=t$,} \\
                          0 \quad&\text{if $s\neq t$.}
                          \end{cases}
\]
The result follows, because by Riemann--Roch $\chi\bigl(\Hom^*(\cU_s,\cU_t) \bigr)=r^2(1-g)=1-d$ for all $s,t\in M$.
 \end{proof}

\subsection{Study of the Orthogonal Complement}\label{subsect:PvdB}

By \cite[Thm.\ B]{PvdB}, there is a semi-orthogonal decomposition of $\D_{\sym_n}(C^n)$ whose pieces are in bijection with the conjugacy classes of $\sym_n$, and which lifts the direct sum decomposition of orbifold cohomology into twisted sectors to the derived category. We want to study how this semi-orthogonal decomposition interacts with the admissible subcategory $I_n(\D(M))$ of \autoref{prop:modulicurves}. For this, we first recall the construction of \cite{PvdB} and then prove that many of the pieces of this semi-orthogonal decomposition are contained in the left-orthogonal complement of $I_n(\D(M))$; see \autoref{prop:complement}.

We write the symmetric power of the curve as $C^{(n)}:=C^n/\sym_n$ and the quotient morphism as $\tau_n\colon C^n\to C^{(n)}$.
Recall that the conjugacy classes of $\sym_n$ are in bijection with partitions of $n$. Let $\lambda=(\lambda_1\ge \lambda_2\ge \dots\ge \lambda_k)$ be such a partition. Then, one representative of the corresponding conjugacy class is
\[
 g_\lambda:=(1,\dots, \lambda_1)(\lambda_1+1,\dots, \lambda_1+\lambda_2)\cdots (\sum_{i=1}^{k-1}\lambda_i +1,\dots,\sum_{i=1}^k\lambda_i)\,.
\]
We denote its fixed-point locus by $C[\lambda]:=(C^n)^{g_\lambda}$. There is an isomorphism $C^k\cong C[\lambda]$ because $C[\lambda]\subset C^n$ is the image of the closed embedding
\begin{equation}\label{eq:delta}
\delta_\lambda:= \delta_{\lambda_1}\times\dots\times\delta_{\lambda_k}\colon C^k\hookrightarrow C^{\lambda_1}\times\dots\times C^{\lambda_k}=C^n
\end{equation}
where the factors $\delta_{\lambda_i}\colon C\hookrightarrow C^{\lambda_i}$ with $\delta_{\lambda_i}(x)=(x,\dots,x)$ are the embeddings of the small diagonals. Concretely, this means
\[
 \delta_\lambda(x_1,\dots,x_k)=\bigl(\underbrace{x_1,\dots,x_1}_{\lambda_1 \text{ times}}, \underbrace{x_2,\dots,x_2}_{\lambda_2 \text{ times}},\dots,  \underbrace{x_k,\dots,x_k}_{\lambda_k \text{ times}}  \bigr)\,.
\]
 We write $\lambda=1^{\mu_1}2^{\mu_2}\cdots \lambda_1^{\mu_{\lambda_1}}$, where $\mu_i$ is the number of times the number $i$ occurs in $\lambda$. For example $(4,3,3,1,1,1)=1^33^24^1=1^32^03^24^1$ in that notation. Then, $\mu_1+\mu_2+\dots+\mu_{\lambda_1}=k$ is the number of parts of $\lambda$, and we consider the subgroup
\[
\sym_\mu=\sym_{\mu_1}\times\dots\times\sym_{\mu_{\lambda_1}}\le \sym_k\,.
\]
Let $\Cent(g_\lambda)\le \sym_n$ be the centraliser of $g_\lambda$. Then there is an isomorphism
\[
 \Cent(g_\lambda)\cong \left(\prod_{i=1}^k Z_{\lambda_i}\right) \rtimes \sym_\mu
\]
where on the right, we have a product of the cyclic groups $Z_{\lambda_i}$ generated by the cycles $g_\lambda$ is composed of, and $\sym_\mu$ acts on this product by permuting the cycles of the same length; see e.g.\ \cite[Prop.\ 1.1.1]{Sagan--book}. The cycles of $g_\lambda$ act trivially on $C[\lambda]=(C^n)^{g_\lambda}$. It follows that the action of $\Cent(g_\lambda)$ on $C[\lambda]\cong C^k$ is given by the action of its quotient $\sym_\mu\le \sym_k$ which permutes the factors. Hence, we have
\begin{equation}\label{eq:quotientdesc}
C_{(\lambda)}:=C[\lambda]/\Cent(g_\lambda)\cong C^{(\mu_1)}\times\dots\times C^{(\mu_{\lambda_1})}\,.
\end{equation}
By the universal property of the quotient $C_{(\lambda)}$, there is a canonical morphism $\nu_\lambda\colon C_{(\lambda)}\to C^{(n)}$. Writing points in the symmetric powers as formal sums of points of $C$, this morphism can be explicitly described under the isomorphism \eqref{eq:quotientdesc} by
\begin{equation}\label{eq:nu}
\begin{aligned}
&\nu_\lambda\bigl(x^1_1+\dots+x^1_{\mu_1}, x^2_1+\dots+x^2_{\mu_2},\dots, x^{\lambda_1}_1+\dots+x^{\lambda_1}_{\mu_{\lambda_1}}  \bigr)\\=& x^1_1+\dots+x^1_{\mu_1}+ 2(x^2_1+\dots+x^2_{\mu_2})+\dots+ \lambda_1(x^{\lambda_1}_1+\dots+x^{\lambda_1}_{\mu_{\lambda_1}})\,.
\end{aligned}
\end{equation}
We define $\olZ_\lambda:=\bigl(C_{(\lambda)}\times_{C^{(n)}} C^n \bigr)_{\mathsf{red}}$ as the reduction of the fibre product defined by the morphisms $\nu_\lambda$ and $\tau_n$. In particular, there is a commutative diagram
\begin{equation}\label{eq:Zdiag}
\begin{tikzcd}
	{\olZ_\lambda} & {C^n} \\
	{C_{(\lambda)}} & {C^{(n)}}\,.
	\arrow["\alpha", from=1-1, to=1-2]
	\arrow["q"', from=1-1, to=2-1]
	\arrow["{\tau_n}", from=1-2, to=2-2]
	\arrow["{\nu_\lambda}", from=2-1, to=2-2]
\end{tikzcd}
\end{equation}
Note that, at first sight, this looks different from the way $\olZ_\lambda$ is defined in the introduction of \cite{PvdB}. However, it is not hard to check that both ways to phrase it lead to the same $\olZ_\lambda$ with the same morphisms to $C_{(\lambda)}$ and $C^n$. The $\sym_n$-action on $C^n$ induces an $\sym_n$-action on $\olZ_\lambda$ such that $\alpha$ is $\sym_n$-equivariant and $q$ is $\sym_n$-invariant.
For later use, note that
\begin{equation}\label{eq:qZ}
\alpha(\olZ_\lambda)=\sym_n\cdot C[\lambda]\subset C^n\,.
\end{equation}
The fully faithful embeddings of \cite{PvdB} are given by the equivariant Fourier--Mukai transforms
\begin{equation}\label{eq:F}
 F_\lambda\coloneqq \FM_{\reg_{\olZ_\lambda}}\cong \alpha_*\circ q^*\circ \triv \colon \D(C_{(\lambda)})\to \D_{\sym_n}(C^n)\,.
\end{equation}
For $\lambda=(n)$, we have $C_{(n)}=C$, and the subvariety $\olZ_{(n)}\subset C\times C^n=C^{n+1}$ equals the small diagonal. Hence, $F_{(n)}=\delta_{n*}\circ\triv$ is the fully faithful functor of \autoref{prop:diagonalff}; see also \autoref{rem:lambdan}.

\begin{theorem}[{\cite[Thm.\ B]{PvdB}}]\label{thm:PvdB}
The functors $F_\lambda$ are fully faithful for every partition $\lambda$ of $n$, and their images form a semi-orthogonal decomposition of $\D_{\sym_n}(C^n)$.
\end{theorem}

By \eqref{eq:F}, the right-adjoint of $F_\lambda$ is
\begin{equation}\label{eq:Rdesc}
R_\lambda= (\_)^{\sym_n}\circ q_* \circ \alpha^!\colon \D_{\sym_n}(C^n)\to \D(C_{(\lambda)})\,.
\end{equation}

\begin{lemma}\label{lem:VB}
Let $E\in \VB_{\sym_n}(C^n)$ be an $\sym_n$-equivariant vector bundle.
\begin{enumerate}
 \item\label{torsionfree} The twisted pull-back $\alpha^!E\in \D(\olZ_\lambda)$ is a torsion free sheaf shifted into degree $n-k$.
\item\label{vb} The image $R_\lambda(E) \in \D(C_{(\lambda)})$ is a vector bundle on $C_{(\lambda)}$ shifted into degree $n-k$.
 \end{enumerate}
\end{lemma}

\begin{proof}
In order to keep the formulas a bit shorter, we abbreviate $Z:=\olZ_\lambda$ and $B:=C_{(\lambda)}$. By \cite[Prop.\ 3.4.3]{PvdB}, $Z$ is Cohen--Macaulay. Hence, there is a torsion-free (see e.g.\ \cite[\href{https://stacks.math.columbia.edu/tag/0AWS}{Tag 0AWS}]{stacks-project} and \cite[\href{https://stacks.math.columbia.edu/tag/0BUS}{Tag 0BUS}]{stacks-project}) dualising sheaf $\omega_Z$ , and we have
\[
 \alpha^!E\cong \alpha^*(E\otimes \omega^{-1}_{C^n})\otimes \omega_Z[\dim(Z)-\dim(C^n)]\cong \sHom(\cE, \omega_Z)[k-n]
\]
where $\cE:=\alpha^*(E\otimes \omega_{C^n}^{-1})^\vee$ is a vector bundle on $Z$.
In particular, this gives \ref{torsionfree}.

Noting that $\omega_Z\cong q^!\omega_B$, since $q$ is finite, we can apply Grothendieck duality along $q$ to obtain
\begin{equation}\label{eq:qfE}
 q_*\alpha^!E\cong q_*\sHom(\cE, q^!\omega_B)[k-n]\cong \sHom(q_*\cE, \omega_B)[k-n]\,.
\end{equation}
Since $q$ is not only finite but also flat (again by \cite[Prop.\ 3.4.3]{PvdB}), the push-forward $q_*\cE$ is again a vector bundle. Hence, \eqref{eq:qfE} gives that $q_*\alpha^! E$ is a vector bundle shifted by $[k-n]$. Since the subsheaf of invariants is always a direct summand, $(q_*\alpha^! E)^{\sym_n}$ is a vector bundle shifted by $[k-n]$ too. By \eqref{eq:Rdesc}, this finishes the proof.
\end{proof}

\begin{prop}\label{prop:complement}
 Let $\lambda=(\lambda_1\ge \lambda_2\ge \dots\ge \lambda_k)$ be a partition of $n$ with $\lambda_1>r$. Then $R_\lambda\circ I_n\cong 0$. In other words, $F_\lambda(\D(C_{(\lambda)}))$ is contained in the left-orthogonal complement of $I_n(\D(M))$.
\end{prop}

\begin{proof}
Vanishing of the functor $R_\lambda\circ I_n$ can be tested on the spanning class of skyscraper sheaves $\reg_t$ for $t\in M$. Let us fix some $t\in M$ and write $U=\cU_t$ to lighten the notation. Then $I_n(\reg_t)=U^{\boxtimes n}$, so we need to show that
\begin{equation}\label{eq:vanish}
 R_\lambda(U^{\boxtimes n})\cong 0\,.
\end{equation}
By \autoref{lem:VB} with $E=U^{\boxtimes n}$, we know that $R_\lambda(U^{\boxtimes n})$ is a shifted vector bundle, so it suffices to prove its vanishing over some non-empty open subset of $C_{(\lambda)}$.

As in the proof of \autoref{lem:VB} above, we write $Z:=\olZ_\lambda$ and $B:=C_{(\lambda)}$. From the
concrete description \eqref{eq:nu} of $\nu_\lambda\colon B\to C^{(n)}$, we see that this morphism is injective on the open locus of points with pairwise distinct $x_i^j$. Hence, there is a non-empty open subset $B^\circ\subset B$ such that $\nu_{\lambda\mid B^\circ}\colon B^\circ\to C^{(n)}$ is an embedding.
We set $Z^\circ:=q^{-1}(B^\circ)$. It follows from the reduced fibre diagram \eqref{eq:Zdiag} that  $\alpha_{\mid Z^\circ}\colon Z^\circ \to C^n$ is an embedding too.

From now on, we identify $Z^\circ$ and $B^\circ$ with their images under these embeddings in $C^n$ and $C^{(n)}$, respectively.
Also, set $C[\lambda]^\circ:=C[\lambda]\cap Z^\circ$. It follows from \eqref{eq:qZ} that $Z^\circ=\sym_n\cdot C[\lambda]^\circ$.
Furthermore, the subgroup
\[
 \sym_{\lambda}=\sym_{\lambda_1}\times \dots\times \sym_{\lambda_k}\le \sym_n
\]
acts trivially on $C[\lambda]$. By \autoref{lem:VB}\ref{torsionfree}, $\alpha^!U^{\boxtimes n}$ is a sheaf concentrated in degree $n-k$. Hence, \autoref{lem:componentvanish} applies, and, together with \eqref{eq:Rdesc}, gives that it suffices to prove the vanishing
\begin{equation}\label{eq:Ccircvanish}\bigl((\alpha^!U^{\boxtimes n})_{\mid C[\lambda]^\circ}\bigr)^{\sym_\lambda}\cong 0\end{equation}
in order to obtain \eqref{eq:vanish}. Under the identification $C[\lambda]\cong C^k$, we have
\[
(\alpha^!U^{\boxtimes n})_{\mid C[\lambda]^\circ}\cong (\delta_\lambda^!U^{\boxtimes n})_{\mid C[\lambda]^\circ}
\]
where $\delta_\lambda=\prod \delta_{\lambda_i}\colon C^k\to C^n$ is the product of small diagonal embeddings as in \eqref{eq:delta}.
Hence, \eqref{eq:Ccircvanish} will be implied by the vanishing $(\delta_\lambda^!U^{\boxtimes n})^{\sym_\lambda}\cong 0$ that we prove now.
We have
\begin{equation}\label{eq:pullbackU}
\delta_\lambda^*U^{\boxtimes n}\cong U^{\otimes \lambda_1}\boxtimes \dots\boxtimes U^{\otimes \lambda_k}
\end{equation}
where the action of $\sym_\lambda=\sym_{\lambda_1}\times \dots\times\sym_{\lambda_k}$ is given by the $i$-th factor $\sym_{\lambda_i}$ of the group acting on the $i$-th factor $U^{\otimes \lambda_i}$ of the box product by permutation of the (inner) tensor factors.

 Furthermore, the normal bundle of the small diagonal $\delta_{\lambda_i}\colon C\hookrightarrow C^{\lambda_i}$ is, as an $\sym_{\lambda_i}$-equivariant bundle, given by $\omega_C^\vee\otimes \rho_{\lambda_i}$ where $\rho_{\lambda_i}$ is the $(\lambda_i-1)$-dimensional standard representation of $\sym_{\lambda_i}$; see \cite[Sect.\ 3]{Kru3}.
Hence, the normal bundle of $\delta_\lambda$ is
\[
 \cN_{\delta_\lambda}\cong (\omega_C^{-1}\otimes \rho_{\lambda_1})\boxplus \dots\boxplus (\omega_C^{-1}\otimes \rho_{\lambda_k})
\]
as an $\sym_\lambda$-equivariant bundle.
Taking the determinant gives the relative dualising bundle
\begin{equation}\label{eq:omegalambda}
 \omega_{\delta_\lambda}\cong \det\cN_{\delta_\lambda}\cong (\omega_C^{1-\lambda_1}\otimes \alt_{\lambda_1})\boxtimes \dots\boxtimes (\omega_C^{1-\lambda_k}\otimes \alt_{\lambda_k})\,,
\end{equation}
where $\alt_{\lambda_i}=\det(\rho_{\lambda_i})$ is the sign representation of $\sym_{\lambda_i}$. Noting that $\codim(\delta_\lambda)=n-k$, Grothendieck duality for regular embeddings together with \eqref{eq:pullbackU} and \eqref{eq:omegalambda} gives
\begin{align*}
\delta_\lambda^!U^{\boxtimes n}&\cong (\delta_\lambda^*U^{\boxtimes n})\otimes \omega_{\delta_\lambda}[k-n]\\&\cong \bigl(U^{\otimes \lambda_1}\otimes \alt_{\lambda_1} \otimes \omega_C^{1-\lambda_1}\bigr) \boxtimes \dots\boxtimes \bigl(U^{\otimes \lambda_k} \otimes \alt_{\lambda_k} \otimes \omega_C^{1-\lambda_k}\bigr)[k-n]\,.
\end{align*}
Taking $\sym_\lambda$-invariants gives
\[
(\delta_\lambda^!U^{\boxtimes n})^{\sym_\lambda}\cong \bigl((\wedge^{\lambda_1} U) \otimes \omega_C^{1-\lambda_1}\bigr) \boxtimes \dots\boxtimes \bigl((\wedge^{\lambda_k} U) \otimes \omega_C^{1-\lambda_k}\bigr)[k-n]
\]
which vanishes if $\lambda_1>r=\rank(U)$ because then $\wedge^{\lambda_1} U=0$.
\end{proof}

\subsection{The Semi-Orthogonal Decomposition of $M_C(2,\Lambda)$}\label{subsect:Lambdasod}
From now on, let $g=g(C)\ge 2$, and $M\coloneqq M_C(2,e)$ with $e$ odd. We denote by $J\coloneqq \Pic^e(C)$ the moduli space of degree $e$ line bundles on $C$, which is non-canonically isomorphic to $\Jac(C)$. There is a smooth and proper morphism
\[
 f\colon M\to J\quad,\quad E\mapsto \det(E)\,.
\]
We fix some $\Lambda\in J$ and denote the corresponding fibre by $M_\Lambda\coloneqq M_C(2,\Lambda)$. It is the fine moduli space of stable bundles of rank $2$ and determinant $\Lambda$ on $C$. Following previous work and conjectures by \cite{Fonarev-Kuz}, \cite{Narasimhan--derived}, \cite{Narasimhan--derivedII}, \cite{Belmans-Muk--admissible}, \cite{Lee--motives}, \cite{BGM--vboncurves}, Tevelev--Torres \cite{TT} and Tevelev \cite{Tevelev--nophantom} constructed a semi-orthogonal sequence of $\D(M_\Lambda)$ with pieces equivalent to derived categories $\D(C^{(\alpha)})$ of symmetric powers of the curve with $0\le \alpha\le g-1$. We recall the construction of the fully faithful embeddings of these pieces in the following, with the goal to extend it to a $J$-linear semi-orthogonal decomposition of $\D(M)$ later in \autoref{subsect:extendsod}.

For a non-negative integer $\alpha$, denote by $\tau_\alpha\colon C^\alpha\to C^{(\alpha)}$ the $\sym_\alpha$-quotient morphism.
For any variety $Y$ and any vector bundle $\cF$ on $C\times Y$, recall that the vector bundle $\cF^{\boxtimes_Y \alpha}$ on $C^\alpha\times Y$ has a canonical $\sym_\alpha$-linearisation; see \autoref{subsect:box}. Hence, by flatness of $\tau_\alpha$, we get on $C^{(\alpha)}\times Y$ the two associated vector bundles
\[
\cF^{(\alpha)}:=\bigl((\tau_\alpha\times \id_Y)_*(\cF^{\boxtimes_Y \alpha})\bigr)^{\sym_\alpha} \quad,\quad \cF^{(\overline \alpha)}:=\bigl((\tau_\alpha\times \id_Y)_*(\cF^{\boxtimes_Y \alpha}\otimes \alt_\alpha)\bigr)^{\sym_\alpha}\,.
\]
In \cite{TT} these bundles are denoted by $\cF^{\boxtimes \alpha}$ and $\overline \cF^{\boxtimes \alpha}$, but we changed the notation to avoid confusion with the box product in the usual sense. Note that, by convention, $\cF^{(0)}\cong \reg_Y$ under the isomorphism $C^{(0)}\times Y\cong \pt\times Y\cong Y$.

Now, let $\cE:=\cE_\Lambda$ be the universal family on $C\times M_\Lambda$. Note that $\theta:=\theta_\Lambda:=\det(\cE_x)\in \Pic(M_\Lambda)$ is independent of the point $x\in C$ because $\Ho^1(\reg_{M_\Lambda})=0$. We assume that $\cE$ is normalised in such a way that $\theta$ is an ample generator of $\Pic(M_\Lambda)\cong \IZ$. That this can always be achieved is explained in \cite[Rem.\ 2.9]{Ramanan--moduli}. We define the functors
\begin{equation}\label{eq:Q}
\begin{aligned}
 &\Psi_\alpha^i:=\FM_{\cQ_\alpha^i}\colon \D(C^{(\alpha)})\to \D(M_\Lambda)\quad,\quad \cQ_\alpha^i:=\cE^{(\alpha)}\otimes(\reg_{C^{(\alpha)}}\boxtimes \theta^{-i})\\
 &\overline \Psi_\alpha^i:=\FM_{\overline \cQ_\alpha^i}\colon \D(C^{(\alpha)})\to \D(M_\Lambda)\quad,\quad \overline \cQ_\alpha^i:=\cE^{(\overline \alpha)}\otimes(\reg_{C^{(\alpha)}}\boxtimes \theta^{-i})\,.
\end{aligned}
\end{equation}
We introduce some notation useful to formulate the next theorem. For $m\in \IZ$, we write
\begin{align*}
 \lfloor m\rfloor_e&:=2\left\lfloor \frac m2\right\rfloor=\max\{k\mid k\le m \text{ even}\}\,,\\
 \lfloor m\rfloor_o&:=2\left\lfloor \frac {m-1}2\right\rfloor+1=\max\{k\mid k\le m \text{ odd}\}\,.
 \end{align*}
\begin{theorem}[\cite{TT}, \cite{Tevelev--nophantom}]\label{thm:TTT}
For every $i\in \IZ$, the functors
\[
 \Psi_0^i,\Psi_2^i,\Psi_4^i,\dots, \Psi_{\lfloor g-1\rfloor_e}^i\quad,\quad \overline\Psi_1^i,\overline\Psi_3^i,\overline\Psi_5^i,\dots, \overline\Psi_{\lfloor g-1\rfloor_o}^i
\]
are fully faithful. Denoting their images in $\D(M_\Lambda)$ by
\[\cG_\alpha^i:=\im(\Psi_\alpha^i)\cong \D(C^{(\alpha)})\quad,\quad \overline \cG_\alpha^i:=\im(\overline \Psi_\alpha^i)\cong \D(C^{(\alpha)})\,,\] there is a semi-orthogonal decomposition
\[
 \D(M_\Lambda)=\left\langle\arraycolsep=0.2cm\def\arraystretch{2.2}\begin{array}{cc}\cG_0^1,\cG_2^2,\cG_4^3, \dots, \cG_{\lfloor g-1\rfloor_e}^{\frac 12\lfloor g-1\rfloor_e+1}, & \overline \cG_{\lfloor g-1\rfloor_o}^{\frac 12\lfloor g-1\rfloor_o+\frac12},\dots , \overline \cG_5^3,  \overline \cG_3^2,  \overline \cG_1^1, \\ \cG_0^0,\cG_2^1,\cG_4^2, \dots, \cG_{\lfloor g-2\rfloor_e}^{\frac 12\lfloor g-2\rfloor_e}, & \overline \cG_{\lfloor g-2\rfloor_o}^{\frac 12\lfloor g-2\rfloor_o-\frac12},\dots , \overline \cG_5^2,  \overline \cG_3^1,  \overline \cG_1^0  \end{array}\right\rangle \,.
\]
\end{theorem}

In particular, $\D(M_\Lambda)$ has a semi-orthogonal decomposition with two pieces isomorphic to $\D(C^{(\alpha)})$ for every $\alpha=0,\dots, g-2$ and one piece isomorphic to $\D(C^{(g-1)})$. The last block $\overline \cG_{\lfloor g-2\rfloor_o}^{\frac 12\lfloor g-2\rfloor_o-\frac12},\dots , \overline \cG_5^2,  \overline \cG_3^1,  \overline \cG_1^0$ of the semi-orthogonal decomposition has to be interpreted as being empty for $g=2$ where $\lfloor g-2\rfloor_o=-1$.

\subsection{Relative Fourier--Mukai Transforms and Semi-Orthogonal Decompositions}\label{subsect:relativeFM}

In order to extend the semi-orthogonal decomposition of \autoref{thm:TTT} from $\D(M_\Lambda)$ to $\D(M)$, we will need a criterion which allows to
test relative semi-orthogonal decompositions fibrewise.

For this, we quickly repeat some basics on relative Fourier--Mukai transforms; for details see e.g.\ \cite[Sect.\ 2.7.--2.9]{Kuz--HPD}, \cite{MR2505443}, or \cite[Sect.\ 3]{BS--descent}.
Throughout this subsection, let $X$, $Y$, and $S$ be smooth projective varieties over an algebraically closed field (here we do not need our general characteristic 0 assumption and can work over arbitrary characteristic) together with smooth morphisms $X\to S$ and $Y\to S$.
Consider the fibre diagram
\[\begin{tikzcd}
	{X\times_S Y } & Y \\
	X & S\,.
	\arrow["{p_Y}", from=1-1, to=1-2]
	\arrow["{p_X}"', from=1-1, to=2-1]
	\arrow[ from=1-2, to=2-2]
	\arrow[from=2-1, to=2-2]
\end{tikzcd}\]
For an object $\cP\in\D(X\times_S Y)$ there is the associated \emph{$S$-relative Fourier--Mukai transform}
\[
 \Phi:=\FM^S_\cP:=p_{Y*}\bigl( \cP\otimes p_X^*(\_)\bigr)\colon \D(X)\to \D(Y)\,.
\]
Every relative Fourier--Mukai transform is (isomorphic to) a special case of a usual Fourier--Mukai transform because $\FM^S_\cP\cong \FM_{i_*\cP}$ where $i\colon X\times_SY\hookrightarrow X\times Y$ is the embedding of the fibre product into the cartesian product.
In particular, every relative Fourier--Mukai transform has a left and a right adjoint; see e.g.\ \cite[Prop.\ 5.9]{Huy}.

For $s\in S$, write $j_s\colon X_s\times Y_s\hookrightarrow X\times_S Y$ for the canonical embedding. We define the induced Fourier--Mukai transform on the fibres by
\[
 \Phi_s:=\FM_{j_s^*\cP}\colon \D(X_s)\to \D(Y_s)\,.
\]
Note that our general smoothness assumptions guarantee that $j_s^*\cP$ is contained in the bounded derived category of $X_s\times Y_s$. In \cite{MR2505443}, relative Fourier--Mukai transforms and their fibres are considered without assuming smoothness of $X$ and $Y$. However, this then requires the authors to often make the additional assumption that $\cP$ is of finite homological dimension over $X$ and $Y$; see \cite[Def.\ 2.1]{MR2505443}. It seems very likely that all the statements below (in particular \autoref{thm:relFMsod}) hold if we drop the smoothness assumptions, but require relative finite homological dimension. However, in our intended application in \autoref{subsect:extendsod}, the total spaces are smooth, so we decided to avoid possible technical difficulties, and not to check whether this generalisation really holds.

\begin{lemma}\label{lem:relFM}
\begin{enumerate}
 \item\label{pushpullcomp} For every $S$-relative Fourier--Mukai transform $\Phi\colon \D(X)\to \D(Y)$ and every $s\in S$, we have isomorphisms of functors
 \[
  j_s^{Y*}\circ \Phi\cong \Phi_s\circ j_s^{X*}\quad,\quad j^Y_{s*}\circ \Phi_s\cong \Phi\circ j^X_{s*}
 \]
where $j_s^X\colon X_s\hookrightarrow X$ and $j_s^Y\colon Y_s\hookrightarrow Y$ are the embeddings of the fibres.
\item\label{composition} Let $Z$ be a smooth projective variety equipped with a smooth morphism to $S$. For any two $S$-relative Fourier--Mukai transforms $\Phi\colon \D(X)\to \D(Y)$ and $\Psi\colon \D(Y)\to \D(Z)$, their composition $\Psi\circ \Phi\colon \D(X)\to\D(Z)$ is again an $S$-relative Fourier--Mukai transform. Furthermore, for every $s\in S$, we have an isomorphism of functors
\[
 (\Psi\circ \Phi)_s\cong \Psi_s\circ \Phi_s\,.
\]
\item\label{rightrelative} For every $S$-relative Fourier--Mukai transform $\Phi\colon \D(X)\to \D(Y)$, its right adjoint $\Phi^R\colon \D(Y)\to \D(X)$ is again an $S$-relative Fourier--Mukai transform. For every $s\in S$, its fibre $(\Phi^R)_s$ is the right-adjoint of $\Phi_s$. In other words $(\Phi^R)_s\cong (\Phi_s)^R$.
\item \label{zero} A relative Fourier--Mukai transform $\Phi\colon \D(X)\to \D(Y)$ is zero if and only if all its fibres are zero:
\[
 \Phi\cong 0\quad\Longleftrightarrow \quad \Phi_s\cong 0 \quad \text{for all $s\in S$.}
\]
\end{enumerate}
\end{lemma}

The analogue of part \ref{rightrelative} for left-adjoints also holds, but we will only need it explicitly for the right-adjoints.

\begin{proof}
Parts \ref{pushpullcomp} and \ref{composition} can be found in \cite[Sect.\ 2.2]{MR2505443} and \cite[Sect.\ 1.3]{Krug--Meachan--univ}. Note that in neither reference detailed proofs are given. However, everything follows from straightforward base-change arguments.

For part \ref{rightrelative}, the right-adjoint $\Phi^R$ is given in \cite[Prop.\ 2.9]{MR2505443} as a relative Fourier--Mukai transform by explicitly naming its kernel on $Y\times_SX$. It is then easy to see that the restriction of this kernel to $Y_s\times X_s$ agrees with the kernel of the right adjoint of $\Phi_s$ as in \cite[Prop.\ 5.9]{Huy}.

For part \ref{zero}, first note that $\Phi=\FM^S_\cP\cong 0$ with $\cP\in \D(X\times_S Y)$ if and only if $\cP\cong 0$. Indeed, it is obvious that $\FM^S_0\cong 0$. In the other direction, assume that $\Phi=\FM^S_\cP\cong 0$. In particular, $\Phi(\reg_x)\cong 0$ for all $x\in X$, which implies that the derived pull-back of $\cP$ to the fibre of $p_X\colon X\times_SY\to X$ over $x$ is zero. Since these fibers cover $X\times_SY$, we get $\cP\cong 0$; compare \cite[Ex.\ 3.30]{Huy}.

Now, since also the fibres of $X\times_SY\to S$ cover $X\times_S Y$, we get the chain of equivalences
\[
\Phi\cong 0\quad\iff\quad \cP\cong 0\quad\iff\quad \cP_s\cong 0\,\,\, \forall \,s\in S \quad\iff\quad \Phi_s\cong 0\,\,\, \forall \,s\in S\,.\qedhere
\]
\end{proof}

\begin{theorem}\label{thm:relFMsod}
\begin{enumerate}
 \item An $S$-relative Fourier--Mukai transform $\Phi\colon \D(X)\to \D(Y)$ is fully faithful if and only if all fibres $\Phi_s\colon \D(X_s)\to \D(Y_s)$ for $s\in S$ are.
 \item Let $X^1,\dots, X^\ell, M$ be smooth projective varieties equipped with smooth morphisms to $S$, and let $\Phi^i\colon \D(X^i)\to \D(M)$ be fully faithful $S$-relative Fourier--Mukai transforms. For $i=1,\dots ,\ell$, write $\cC^i=\Phi^i(\D(X^i))\subset \D(M)$ and $\cC^i_s=\Phi^i_s(\D(X^i_s))\subset \D(M_s)$. Then
\begin{equation}\label{eq:Msodgeneral}
 \D(M)=\langle \cC^1,\dots ,\cC^\ell\rangle
\end{equation}
 is a semi-orthogonal decomposition if and only if
\begin{equation}\label{eq:Mssod}
 \D(M_s)=\langle \cC^1_s,\dots ,\cC^\ell_s\rangle
\end{equation}
 is a semi-orthogonal decomposition for all $s\in S$.
\end{enumerate}
\end{theorem}

\begin{proof}
 Part (i) is \cite[Prop.\ 2.15]{MR2505443}.
It seems very likely that part (ii) can be deduced as a special case of \cite[Thm.\ A]{BS--descent}, using the conservative functor $F\colon \D(M)\to \prod_{s\in S} \D(M_s)$ given by the pull-backs to the fibres, but there are some potential technical issues to check. Hence, we give an elementary proof using \autoref{lem:relFM}.

For the semi-orthogonality, let $1\le i<j\le \ell$. Then, $\Hom(\cC^j,\cC^i)=0$ if and only if the composition $(\Phi^j)^R\circ \Phi^i$ vanishes. Combining parts \ref{composition}, \ref{rightrelative}, and \ref{zero} of \autoref{lem:relFM}, this is equivalent to $(\Phi^j_s)^R\circ \Phi^i_s\cong 0$ for every $s\in S$, which in turn is equivalent to $\Hom(\cC_s^j,\cC_s^i)=0$.

Since all the subcategories $\cC^i$ are admissible, fullness of \eqref{eq:Msodgeneral}, which means explicitly that
$\langle \cC^1,\dots ,\cC^\ell\rangle$ is all of $\D(M)$, is equivalent to the right-orthogonal $\langle \cC^1,\dots ,\cC^\ell\rangle^\perp$ vanishing. This is equivalent to the condition that, for $E\in \D(M)$, we have $(\Phi^i)^R(E)\cong 0$ for all $i=1,\dots,\ell$ only if already $E\cong 0$.

Let us now assume that the semi-orthogonal decomposition \eqref{eq:Mssod} is full for every $s\in S$, and let $E\in \D(M)$ with $(\Phi^i)^R(E)\cong 0$ for all $i=1,\dots,\ell$ be given. Then, by parts \ref{pushpullcomp} and \ref{rightrelative} of \autoref{lem:relFM} together, we have $0\cong j_s^*(\Phi^i)^R(E)\cong (\Phi^i_s)^R(j_s^* E)$ for all $s\in S$. By fullness of \eqref{eq:Mssod}, this implies $j_s^* E\cong 0$. Since the embeddings $j_s\colon M_s\hookrightarrow M$ cover $M$, this implies $E\cong 0$ as desired; compare again \cite[Ex.\ 3.30]{Huy}.

The proof that the fullness of \eqref{eq:Msodgeneral} implies the fullness of \eqref{eq:Mssod} for every $s\in S$ is very similar using the compatibility with push-forwards of \autoref{lem:relFM}\ref{pushpullcomp} instead of the compatibility with pull-backs.
\end{proof}

\begin{remark}\label{rem:reltrick}
Let still $X$, $Y$, and $S$ be smooth projective varieties and $f\colon Y\to S$ a smooth morphism. Even if no morphism $X\to S$ is given,
by a simple trick, we can turn any $\cP\in \D(X\times Y)$ into a kernel of an $S$-relative Fourier--Mukai transform. Namely, there is a canonical isomorphism
\[
 X\times Y\xrightarrow\cong (X\times S)\times_S Y\quad,\quad (x,y)\mapsto (x,f(y),y)\,.
\]
This allows us to consider $\cP$ as an object in $\D\bigl((X\times S)\times_S Y\bigr)$ and define the $S$-relative Fourier--Mukai transform
\[
\Phi:=\FM_{\cP}^S\colon \D(X\times S)\to \D(Y)\,.
\]
Since the diagram
\[\begin{tikzcd}
	{X\times Y_s} && {(X\times S)_s\times Y_s} \\
	{X\times Y} && {(X\times S)\times_S Y}
	\arrow["\cong", from=1-1, to=1-3]
	\arrow["{\iota_s}"', hook, from=1-1, to=2-1]
	\arrow["{j_s}", hook, from=1-3, to=2-3]
	\arrow["\cong", from=2-1, to=2-3]
\end{tikzcd}\]
where the vertical arrows are the canonical embeddings commutes for every $s\in S$, the fibers of $\Phi$ are given by
\[
 \Phi_s\cong \FM_{\cP_s}\colon \D(X)\to \D(Y_s)\quad\text{where}\quad \cP_s:=\iota_s^*\cP\,.
\]
\end{remark}

\subsection{Extension to a Semi-Orthogonal Decomposition of $M_C(2,e)$}\label{subsect:extendsod}

Let again $M=M_C(2,e)$ with $e$ odd. The universal family $\cU\in \VB(C\times M)$ is only unique up to twist by line bundles on $M$. We have to choose it carefully to guarantee that its restrictions to the fibres $M_\Lambda$ of $f\colon M\to J$ satisfy the assumptions of \autoref{thm:TTT} as laid out in \autoref{subsect:Lambdasod}.

\begin{lemma}\label{lem:Uchoice}
 The universal family $\cU$ can be chosen such that
 $\det \cU_{\mid \{x\}\times M_\Lambda}\cong \theta_\Lambda$ for every $x\in C$ and every $\Lambda\in J$. In other words, $\det \cU_{\mid \{x\}\times M_\Lambda}$ is the unique ample generator of $\Pic(M_\Lambda)$.
\end{lemma}

\begin{proof}
 Let us fix some $\Lambda_0\in J$ and write $M_0:=M_{\Lambda_0}$. As pointed out in \autoref{subsect:Lambdasod}, there is a choice of the universal family $\cE_0\in \VB(C\times M_0)$ such that $\det \cE_{0\mid x\times M_0}\cong \theta_0$ is the ample generator of $\Pic(M_0)$.

By the universal property of $M$, starting with any choice of a universal family $\cU$, we have that $\cU_{\mid C\times M_0}$ differs from $\cE_0$ only by twist by a line bundle on $M_0$. Since the restriction map $\Pic(M)\to \Pic(M_0)$ is surjective \cite[Thm.\ B]{DN}, we can change the choice of $\cU$ if necessary to achieve
$\cU_{\mid C\times M_0}\cong \cE_0$. Hence, $\det \cU_{\mid \{x\}\times M_0}\cong \theta_0$, and we need to show that the same holds for $\Lambda_0$ replaced by any other degree $e$ line bundle $\Lambda$.
For this, we consider the degree $2^{2g}$ \'etale morphism
\[
 u\colon \Jac(C)\times M_0\to M\quad,\quad (L,E)\mapsto L\otimes E\,;
\]
see e.g.\ \cite[Sect.\ 2]{MR1881021}. Its restriction to $\{\reg_C\}\times M_0$ is a closed embedding with image $M_0$. Hence, fixing some $x\in C$ and writing $\Theta_x:=\det(\cU_x)\in \Pic(M)$ under the identification $\{x\}\times M=M$, we have
that $(u^*\Theta_x)_{\mid \{0\}\times M_0}\cong \theta_0$ is the unique ample generator of $\Pic(M_0)$. Now for any degree $e$ line bundle $\Lambda$ on $C$, there is an $L\in \Jac(C)$ with $\Lambda\cong L^2\otimes \Lambda_0$. The restriction of $u$ to $\{L\}\times M_0$ is a closed embedding with image $M_\Lambda$. Since $\Ho^1(\reg_{M_0})=0$, the line bundle $u^*\Theta_x$ does not vary along the fibres over $\Jac(C)$. Hence, $(u^*\Theta_x)_{\mid \{L\}\times M_0}\cong \theta_0$, which implies that $\Theta_{x \mid M_\Lambda}$ is the unique ample generator of $\Pic(M_\Lambda)$ for every $\Lambda\in J$.
\end{proof}

Let now $\cU\in \VB(C\times M)$ be the universal family as in \autoref{lem:Uchoice}, and write \[\Theta:=\Theta_x:=\det(\cU_x)\in \Pic(M)\] as in the proof above. Note that different choices of $x$ might lead to the resulting $\Theta_x$ differing by the pull-back of a line bundle on $J$ along $f\colon M\to J$. However, for any choice, the restriction of $\Theta$ to $M_\Lambda$ will be the unique ample generator of $\Pic(M_\Lambda)$ for every $\Lambda\in J$.

For integers $\alpha$ and $i$ with $\alpha\ge 0$, we consider the vector bundles
\begin{equation}\label{eq:R}
 \cR_\alpha^i=\cU^{(\alpha)}\otimes \pr_M^*\Theta^{-i} \quad,\quad \overline \cR_\alpha^i=\cU^{(\overline \alpha)}\otimes \pr_M^*\Theta^{-i}
\end{equation}
on $C^{(\alpha)}\times M$. As explained in \autoref{rem:reltrick}, using the canonical isomorphism \[C^{(\alpha)}\times M\cong (C^{(\alpha)}\times J)\times_J M\,,\] we can consider them as kernels of the $J$-relative Fourier--Mukai transforms
\[
 \Phi_\alpha^i:=\FM^J_{\cR_\alpha^i} \colon \D(C^{(\alpha)}\times J)\to \D(M)\quad,\quad \overline \Phi_\alpha^i:=\FM^J_{\overline \cR_\alpha^i} \colon \D(C^{(\alpha)}\times J)\to \D(M)
\]
The kernels $\cR_\alpha^i$ look very similar to the $\cU^{\lhd n}$ used to define $I_n$. In particular, they also involve box products of $\cU$. Note, however, that in $\cR_\alpha^i$ we push down to the quotient variety and take invariants, while $\cU^{\lhd n}$ stays on the cartesian product. Hence, as an $\sym_n$-equivariant object, $\cU^{\lhd n}$ can be regarded as a sheaf on the quotient \emph{stack}, not the quotient variety.
Furthermore, the associated Fourier--Mukai transforms that we consider go in different directions. While $I_n=\FM_{\cU^{\lhd n}}$ has $\D(M)$ as the source,
$\Phi_\alpha^i=\FM^J_{\cR_\alpha^i}$ has $\D(M)$ as the target.

\begin{theorem}\label{thm:Msod}
For every $i\in \IZ$, the functors
\[
 \Phi_0^i,\Phi_2^i,\Phi_4^i,\dots, \Phi_{\lfloor g-1\rfloor_e}^i\quad,\quad \overline\Phi_1^i,\overline\Phi_3^i,\overline\Phi_5^i,\dots, \overline\Phi_{\lfloor g-1\rfloor_o}^i
\]
are fully faithful. Denoting their images in $\D(M)$ by
\[\cC_\alpha^i:=\im(\Phi_\alpha^i)\cong \D(C^{(\alpha)}\times J)\quad,\quad \overline \cC_\alpha^i:=\im(\overline \Phi_\alpha^i)\cong \D(C^{(\alpha)}\times J)\,,\] there is a semi-orthogonal decomposition
\begin{equation}\label{eq:Msod}
 \D(M)=\left\langle\arraycolsep=0.2cm\def\arraystretch{2.2}\begin{array}{cc}\cC_0^1,\cC_2^2,\cC_4^3, \dots, \cC_{\lfloor g-1\rfloor_e}^{\frac 12\lfloor g-1\rfloor_e+1}, & \overline \cC_{\lfloor g-1\rfloor_o}^{\frac 12\lfloor g-1\rfloor_o+\frac12},\dots , \overline \cC_5^3,  \overline \cC_3^2,  \overline \cC_1^1, \\ \cC_0^0,\cC_2^1,\cC_4^2, \dots, \cC_{\lfloor g-2\rfloor_e}^{\frac 12\lfloor g-2\rfloor_e}, & \overline \cC_{\lfloor g-2\rfloor_o}^{\frac 12\lfloor g-2\rfloor_o-\frac12},\dots , \overline \cC_5^2,  \overline \cC_3^1,  \overline \cC_1^0  \end{array}\right\rangle\,.
\end{equation}
\end{theorem}

In particular, $\D(M)$ has a semi-orthogonal decomposition with two pieces isomorphic to $\D(C^{(\alpha)}\times J)$ for every $\alpha=0,\dots, g-2$ and one piece isomorphic to $\D(C^{(g-1)}\times J)$.

\begin{proof}
For every $\Lambda\in J$, the restrictions of the kernels $\cR_\alpha^i$ and $\overline\cR_\alpha^i$ in \eqref{eq:R} to $C^{(\alpha)}\times M_\Lambda$ are the kernels $\cQ_\alpha^i$ and $\overline\cQ_\alpha^i$ in \eqref{eq:Q} which define the fully faithful functors and the semi-orthogonal decomposition of \autoref{thm:TTT}. Hence, \autoref{thm:relFMsod} applies to give the assertion.
\end{proof}

\subsection{Possible Connection Between the Semi--Orthogonal Decompositions}\label{subsect:possibleconnection}

We will also say a few words on the case of higher rank $r> 2$ later in this subsection, but for now let still $M:=M_C(2,e)$ and $n\ge d=4(g-1)+1$.
There is some hope that the embedding $I_n\colon \D(M)\hookrightarrow \D_{\sym_n}(C^n)$ of \autoref{prop:modulicurves} can be used to give an alternative proof of \autoref{thm:Msod}, and then, looking at the fibres over $J$, also of \autoref{thm:TTT}. Besides pure curiosity, this would be desirable because the  proof of \autoref{thm:TTT}, as given in \cite{TT} and \cite{Tevelev--nophantom}, is technically very difficult.

To simplify the discussion, let us choose $n$ minimal, which means $n=d=4(g-1)+1$. In \autoref{prop:complement}, we already proved that all parts
$\D(C_{(\lambda)})$ of the semi--orthogonal decomposition of $\D_{\sym_d}(C^d)$ of Polishchuk and van den Bergh \cite{PvdB} with $\lambda_1>2$ are contained in the left-orthogonal complement of $I_d(\D(M))$. Hence, $I_d(\D(M))$ is contained in the part of the semi-orthogonal decomposition of
Polishchuk and van den Bergh spanned by $\D(C_{(\lambda)})$ with $\lambda=1^{d-2\alpha}\,2^\alpha$ for $\alpha=0,1,\dots, 2(g-1)$.

Note that for every piece of the semi-orthogonal decomposition \eqref{eq:Msod} of $\D(M)$, equivalent admissible subcategories of these $\D(C_{(\lambda)})$ exist. Indeed, all pieces of \eqref{eq:Msod} are equivalent to $\D(C^{(\alpha)}\times J)$ with $\alpha\le g-1$. Furthermore, for
$\lambda=1^{d-2\alpha}\,2^\alpha$, we have $C_{(\lambda)}=C^{(\alpha)}\times C^{(d-2\alpha)}$. For $\alpha\le g-1$, we have
\[
 d-2\alpha=4(g-1)+1-2\alpha\ge 2g-1\,.
\]
Hence, $d-2\alpha$ is in the range where the Abel-Jacobi map $\AJ_{d-2\alpha}\colon C^{(d-2\alpha)}\to J$ is a $\IP$-bundle. By \cite[Sect.\ 2]{Orlov--bundles}, this gives a decomposition of $\D(C^{(d-2\alpha)})$ into $d-2\alpha-g+1$ pieces equivalent to $\D(J)$. Hence, $\D(C_{(\lambda)})\cong \D(C^{(\alpha)}\times C^{(d-2\alpha)})$ contains as many pieces equivalent to $\D(C^{(\alpha)}\times J)$.

Unfortunately, after some tests, it does not appear that any of the obvious ways to realise these pieces in $\D(C_{(\lambda)})$, namely as base changes of some line bundle twist of $\AJ^*_{d-2\alpha}(\D(J))$, gives pieces of $F_\lambda(\D(C_{(\lambda)}))\subset \D_{\sym_d}(C^d)$ which are contained in $I_d(\D(M))$. Hence, deducing \autoref{thm:Msod} from \autoref{thm:PvdB} and the embedding $I_d\colon \D(M)\hookrightarrow\D_{\sym_d}(C^d)$ will probably not be easy either.

For higher rank, while there are no complete results as in \cite{TT} and \cite{Tevelev--nophantom} for $r=2$, there still is a great body of work on semi--orthogonal decompositions of $\D\bigl(M_C(r,\Lambda)\bigr)$ for $r>2$, see \cite{Belmans-Muk--admissible}, \cite{Gomez-Lee--motivic}, \cite{Lee-Moon--acm}, \cite{Lee-Moon--sym}, \cite{Xu--modulisod}.
Since we have the fully faithful functors $I_n\colon \D(M_C(r,e))\hookrightarrow \D_{\sym_n}(C^n)$ of \autoref{prop:modulicurves} also for higher $r$, a similar approach to the one alluded to above could yield new proofs of the known results, and potentially even new results, for $r>2$.

\section{Proof of Necessity for Full Faithfulness}\label{sect:necessary}

As before, let $M$ and $X$ be smooth projective varieties with $d:=\dim M$. Furthermore, $\cU\in \D(M\times X)$ is an object, $\cU^{\lhd n}\in \D_{\sym_n}(M\times X^n)$ the box product over $M$, and \[I_n=\FM_{\cU^{\lhd n}}\colon \D(M)\to \D_{\sym_n}(X^n)\,.
\]
\begin{theorem}\label{thm:exactcriterion}
Consider the following subset of the positive integers:
\[
 \cJ:=\bigl\{ n\in \IZ\mid n\ge 2\text{ and } I_n\colon \D(M)\to \D_{\sym_n}(X^n) \text{ is fully faithful }\bigr\}
 \]
The following three conditions are equivalent:
\begin{enumerate}
 \item\label{i} the set $\cJ$ is non-empty,
\item we have $\cJ=\bigl[\max\{d,2\},\infty\bigr)$,
\item the graded Hom spaces between the fibers of $\cU$ satisfy
\begin{equation}\label{eq:exactcond}
 \Hom^*(\cU_s,\cU_t)\cong \begin{cases}
                          \kk[0]\oplus \kk^{\oplus d}[-1]\quad&\text{if $s=t$,} \\
                          \kk^{\oplus d-1}[-1]\quad&\text{if $s\neq t$.}
                          \end{cases}
\end{equation}
\end{enumerate}
\end{theorem}

From now on, assume that $I_n$ is fully faithful for some $n\ge 2$, which means that \ref{i} holds. In the remainder of this section, we will prove that $n\ge d$ and that condition \eqref{eq:exactcond} is fulfilled. Together with \autoref{thm:mainintro}, which we already proved in \autoref{sect:sufficient}, this gives \autoref{thm:exactcriterion}.

As in \autoref{sect:sufficient}, we set $\cW:=\cU^{\lhd n}$.

\begin{lemma}\label{lem:selfHom}
For every $t\in M$, we have $\Hom^*(\cU_t, \cU_t)\cong \kk[0]\oplus \kk^d[-1]$.
Furthermore, $n\ge d$ holds.
\end{lemma}
\begin{proof}
By full faithfulness of $I_n$ together with \eqref{eq:boxfibre} and \autoref{lem:symKun},
\begin{equation}\label{eq:symHom}
\wedge^* T_M(t)\cong  \Hom^*(\reg_t,\reg_t)\cong \Hom^*_{\sym_n}(\cW_t,\cW_t)\cong S^n\Hom^*(\cU_t,\cU_t)\,.
\end{equation}
In particular, $\Hom^*(\cU_t,\cU_t)$ does not have a non-vanishing part of negative degree. Indeed, assume that there was some homogeneous $0\neq \gamma\in\Hom^*(\cU_t,\cU_t)$ of negative degree. Then also $\gamma\cdot \id^{n-1}\in S^n\Hom^*(\cU_t,\cU_t)$ would be non-zero and of negative degree, contradicting \eqref{eq:symHom}.

Now, because of the vanishing of $\Hom^*(\cU_t,\cU_t)$ in negative degrees, evaluating \eqref{eq:symHom} in degree zero gives $\kk\cong \Hom^0(\reg_t,\reg_t)\cong S^n\Hom^0(\cU_t,\cU_t)$. This implies
\[
 \Hom^0(\cU_t,\cU_t)\cong \kk\,.
\]
Using this information to evaluate \eqref{eq:symHom} in degree 1 then gives
\[
\kk^d \cong T_M(t)\cong  S^{n-1}\bigl(\Hom^0(\cU_t,\cU_t)\bigr)\otimes \Hom^1(\cU_t,\cU_t)\cong  \kk\otimes \Hom^1(\cU_t,\cU_t)\cong \Hom^1(\cU_t,\cU_t)\,.
\]
Now, let us write
\begin{equation}\label{eq:Udecomp}\Hom^*(\cU_t,\cU_t)\cong \kk[0]\oplus V[-1]\oplus R\,,\end{equation}
where $V\cong T_M(t)\cong \kk^d$, and $R$ is a graded vector space concentrated in degrees other than $0$ and $1$.
We consider the graded subspace
\begin{equation}\label{eq:subspace}
\bigoplus_{i=0}^{\min(n,d)} \wedge^iV[-i]\cong  S^n(\kk[0]\oplus V[-1])\subset S^n\Hom^*(\cU_t,\cU_t)\cong \wedge^*V
\end{equation}
where the first isomorphism comes from \autoref{rem:graded}, and the last isomorphism is \eqref{eq:symHom} again. Since $n\ge 2$, this inclusion is an equality in degree 2. Hence, $R$ is zero in degree 2, because otherwise we would have an additional contribution to the right-hand side of \eqref{eq:subspace}. It follows that the multiplication of two degree 1 elements in $\Hom^*(\cU_t,\cU_t)$ is always zero. Hence, the left-hand side of \eqref{eq:subspace} agrees with the subalgebra generated by the degree 1 part of
$S^n\Hom^*(\cU_t,\cU_t)$.

Note that \eqref{eq:symHom} is a chain of isomorphisms of graded algebras. In particular, the algebra $S^n\Hom^*(\cU_t,\cU_t)$ is generated in degree one.
Hence, the inclusion \eqref{eq:subspace} must be an equality.
This proves both assertions, that $n\ge d$, and that $R=0$ in \eqref{eq:Udecomp}.
\end{proof}

It is only left to prove that $\Hom^*(\cU_s,\cU_t)\cong \kk^{\oplus d-1}[-1]$ for $s\neq t$.

\begin{lemma}\label{lem:odd}
For $s,t\in M$ with $s\neq t$, the graded vector space $\Hom^*(\cU_s,\cU_t)$ is concentrated in odd degrees.
\end{lemma}

\begin{proof}
 By full faithfulness of $I_n$, we have $\Hom^*_{\sym_n}(\cW_s, \cW_t)\cong \Hom^*(\reg_s,\reg_t)\cong 0$. On the other hand, writing $V:=\Hom^*(\cU_s,\cU_t)=V^{\mathsf{odd}}\oplus V^{\mathsf{even}}$ as in \autoref{rem:graded}, we have
 \[
 0\cong \Hom^*_{\sym_n}(\cW_s, \cW_t)\cong S^nV\cong \bigoplus_{i=0}^n S^{n-i} V^{\mathsf{even}}\otimes S^{i} V^{\mathsf{odd}}\,.
 \]
Hence, all the summands on the right-hand side need to vanish. Because $S^nV^{\mathsf{even}}$ is the symmetric power for the underlying non-graded vector space, the $i=0$ summand can only vanish if $V^{\mathsf{even}}\cong 0$.
 \end{proof}

\begin{proof}[Finishing the proof of \autoref{thm:exactcriterion}]
It is left to show that $\Hom^*(\cU_s,\cU_t)\cong \kk^{\oplus d-1}[-1]$ for all $(s,t)\in M\times M$ with $s\neq t$.
As in the proof of \autoref{prop:surfcondition}, we consider
\[
 \cA:=\pr_{13*}\sHom(\pr_{12}^* \cU, \pr_{23}^*\tau^* \cU)\,\in \D(M\times M)
\]
where $\tau\colon X\times M\to M\times X$ is the morphism which switches the factors, and $\pr_{ij}$ denotes the projection from $M\times X\times M$ to pairs of factors.
Again,
\begin{equation}\label{eq:Afibres}
 \iota_{s,t}^*\cA\cong \Hom^*(\cU_s,\cU_t)
\end{equation}
for the embedding $\iota_{s,t}\colon\pt\to M\times M$ of any point $(s,t)\in M\times M$. By \autoref{lem:selfHom}, we have $\chi\bigl(\Hom^*(\cU_t,\cU_t)\bigr)=1-d$. By \eqref{eq:Afibres} together with \autoref{lem:Eulerconst}, then also
\[\chi\bigl(\Hom^*(\cU_s,\cU_t)\bigr)=1-d\,.\] Since $\Hom^*(\cU_s,\cU_t)$ is concentrated in odd degrees by \autoref{lem:odd}, this gives
\begin{equation}\label{eq:homdim}
 \dim \Hom^*(\cU_s,\cU_t)=d-1 \quad \text{for all $(s,t)\in M\times M$ with $s\neq t$.}
\end{equation}
By semi-continuity \cite[\href{https://stacks.math.columbia.edu/tag/0BDI}{Lemma 0BDI}]{stacks-project}, for every $i\in \IZ$, the map
\[
 M\times M\to \IZ\quad,\quad (s,t)\mapsto \mathsf{h}^i(\iota_{s,t}^*\cA)=\hom^i(\cU_s,\cU_t)
\]
is upper semi-continuous. Since we know the vanishing of $\hom^i(\cU_t,\cU_t)$ for every $t\in M$ and every odd $i\neq 1$ by \autoref{lem:selfHom}, we get open neighbourhoods $W_i$ of the diagonal in $M\times M$ such that $\hom^i(\cU_s,\cU_t)=0$ for all $(s,t)\in W_i$. By \autoref{lem:Eulerconst}(i), all but finitely many $W_i$ can be chosen as the whole $M\times M$. Hence, the intersection
\[W:=\bigcap_{\substack{i\text{ odd}\\i\neq 1}} W_i\] is still a non-empty open subset of $M\times M$. By \eqref{eq:homdim}, we have $\hom^1(\cU_s,\cU_t)=d-1$ for all $(s,t)\in W$ with $s\neq t$. Hence, by semi-continuity for $i=1$, we get that
\[
 \hom^1(\cU_s,\cU_t)\ge d-1\quad\text{for all $(s,t)\in M\times M$.}
\]
By \eqref{eq:homdim} this inequality must be an equality for $s\neq t$. Again by \eqref{eq:homdim}, this implies the vanishing of $\Hom^i(\cU_s,\cU_t)$ for all $i\neq 1$ and all $(s,t)\in M\times M$ with $s\neq t$.
\end{proof}

\bibliographystyle{alpha}
\addcontentsline{toc}{chapter}{References}
\bibliography{references}

\end{document}